\documentclass{article}

\usepackage{hyperref}
\usepackage{lipsum}
\usepackage{amsfonts}
\usepackage{graphicx}
\usepackage{float}
\usepackage{epstopdf}
\usepackage{algorithmic}
\usepackage{amsmath}
\usepackage{amsthm}
\usepackage{mathtools}
\usepackage{xcolor} 
\usepackage{subcaption} 
\usepackage{bbm} 
\usepackage[multiple]{footmisc}
\usepackage[a4paper, total={16cm, 20cm}]{geometry}

\ifpdf
  \DeclareGraphicsExtensions{.eps,.pdf,.png,.jpg}
\else
  \DeclareGraphicsExtensions{.eps}
\fi

\usepackage{enumitem}
\setlist[enumerate]{leftmargin=.5in}
\setlist[itemize]{leftmargin=.5in}

\ifpdf
\hypersetup{
  pdftitle={Matrix Rigidity & Ill-Posedness of Robust PCA},
  pdfauthor={J. Tanner, A. Thompson, and S. Vary}
}
\fi

\newtheorem{remark}{Remark}[section]
\newtheorem{conjecture}{Conjecture}[section]

\newtheorem{lemma}{Lemma}[section]
\newtheorem{theorem}{Theorem}[section]

\numberwithin{equation}{section} 

\title{Matrix rigidity and the ill-posedness of \\ Robust~PCA and matrix~completion\footnote{This publication is based on work partially supported by: the EPSRC I-CASE studentship (voucher 15220165) in partnership with Leonardo, The Alan Turing Institute through EPSRC (EP/N510129/1) and the Turing Seed Funding grant SF019, Institut Henri Poincar\'e (UMS 839 CNRS-Sorbonne Universit\'e), and LabEx CARMIN (ANR-10-LABX-59-01).}
}

\date{July 11, 2019}

\author{    Jared~Tanner\footnote{Mathematical Institute, University of Oxford, Woodstock Road, Oxford OX2 6GG, UK}
		\footnote{The Alan Turing Institute, British Library, London NW1 2DB, UK}
	\and Andrew~Thompson\footnote{National Physical Laboratory, Hampton Road, Teddington TW11 0LW, UK}
	\and Simon~Vary\footnotemark[2]
}

\usepackage{amsopn}

\newcommand{\pR}{\mathbb{R}}

\DeclareMathOperator{\LS}{LS}
\newcommand{\PO}{P_\Omega}
\DeclareMathOperator{\Rig}{Rig}
\newcommand{\bigO}{\mathcal{O}}
\DeclareMathOperator{\rank}{rank}

\newcommand{\ceil}[1]{\left\lceil #1 \right\rceil}
\newcommand{\MO}{M_n} 
\newcommand{\MOS}{M_n(\epsilon)} 
\newcommand{\MT}{N_n} 
\newcommand{\MTS}{N_n(\epsilon)} 
\newcommand{\MTN}{E} 
\newcommand{\MTH}{\hat{M}_{n'}} 
\newcommand{\MTHS}{\hat{M}_{n'}(\epsilon)} 

\begin{document}

\maketitle

\begin{abstract}
Robust Principal Component Analysis (PCA) (Cand{\`{e}}s et al., 2011)
and low-rank matrix completion (Recht et al., 2010) are extensions of
PCA that allow for outliers and missing entries respectively.
It is well-known that solving these problems requires a low coherence between the low-rank matrix and the canonical basis, since in the extreme cases -- when the low-rank matrix we wish to recover is also sparse -- there is an inherent ambiguity.
However, the well-posedness issue in both problems is an even more
fundamental one: in some cases, both Robust PCA and matrix completion
can fail to have any solutions due to the set of low-rank plus
sparse matrices not being closed, which in turn is equivalent to the
notion of the matrix rigidity function not being lower semicontinuous
(Kumar et al., 2014). By constructing infinite families of matrices,
we derive bounds on the rank and sparsity such that the set of
low-rank plus sparse matrices is not closed. We also demonstrate
numerically that a wide range of non-convex algorithms for both Robust
PCA and matrix completion have diverging components when applied to
our constructed matrices.  
An analogy can be drawn to the case of sets of higher order tensors
not being closed under canonical polyadic (CP) tensor rank, rendering
the best low-rank tensor approximation unsolvable (De Silva and Lim,
2008) and hence encourage the use of multilinear tensor rank (De
Lathauwer, 2000). 
\end{abstract}



\section{Introduction}
Principal Component Analysis (PCA) plays a crucial role in the analysis of high-dimensional data \cite{Vaswani2018rethinking, Ringner2008what, Abdi2010principal, Jolliffe1988principal} and is a widely used dimensionality reduction technique \cite{Han2017clustering,Hsu2016accumulation,Plesa2018multiplexed, Miehlbradt2018data}.
It involves solving a low-rank approximation which can be easily computed for moderate size problems \cite{Demmel1997applied} by computing the singular value decomposition (SVD), or for larger problem sizes using notions of sketching to compute leading portions of the SVD \cite{Halko2011finding,Drineas2006fast,Woodruff2014sketching}. 
Over the last decade PCA has been extended to allow for missing data (matrix completion) or data with either corrupted or few entries inconsistent with a low-rank model (Robust PCA).  In this manuscript we show that the set of matrices which are the sum of low-rank and sparse matrices is not closed for a range of rank, sparsity, and matrix dimensions; see Theorem \ref{thm:main_intro}.   Moreover there are a number of algorithms that, when given a matrix of a specific form and with constraints on the rank and sparsity, seek such a decomposition where the constituents diverge while at the same time the sum of the matrices converges to a bounded matrix outside of the feasible set of prescribed rank and sparsity, see Section \ref{sec:numerics}.  We thereby highlight a previously unknown issue practitioners might experience using these techniques. The situation is analogous to the lack of closedness for Tensor CP decomposition rank \cite{Hitchcock1928multiple,Hitchcock1927the} which motivates the notions of multilinear rank approximation~\cite{Lathauwer2000on}.

\subsection{Prior work}
Robust PCA (RPCA) solves a low-rank plus sparse matrix approximation with the sparse component allowing for few but arbitrarily large corruptions in the low-rank structure; that is, a matrix $M\in\pR^{m\times n}$ is decomposed into a low-rank matrix $L$ plus a sparse matrix $S$ 
\begin{equation}\label{eq:RPCA_problem}
	\min_{X\in \pR^{m\times n}} \|X-M\|_F, \quad \text{s.t.}\quad X\in \LS_{m,n}(r,s),
\end{equation}
where $\LS_{m,n}(r,s)$ is the set of $m\times n$ matrices that can be expressed as a rank $r$ matrix $L$ plus a sparsity $s$ matrix $S$
\begin{equation*}\label{eq:LS_def}
	\LS_{m,n}(r,s) = \left\{ L + S\in\pR^{m\times n}: \, \rank{(L)}\leq r, \, \|S\|_0 \leq s \right\}.
\end{equation*}
We omit the subscript and write $\LS(r,s)$ where the matrix size is implied from the context and use only a single subindex $\LS_n(r,s)$ to denote sets of square matrices $\LS_{n,n}(r,s)$.   Allowing the addition of a sparse matrix to the low-rank matrix can be viewed as modelling globally correlated structure in the low-rank component while allowing local inconsistencies, innovations, or corruptions.  Exemplar applications of this model include image restoration \cite{Gu2014weighted}, hyperspectral image denoising \cite{Gogna2014split, Chen2017denoising, Wei2016hyperspectral}, face detection \cite{Luan2014extracting, Wright2009robust}, acceleration of dynamic MRI data acquisition \cite{Otazo2015low, Xu2017dynamic}, analysis of medical imagery \cite{Baete2018low, Gao2011robust}, separation of moving objects in at otherwise static scene \cite{Bouwmans2016decomposition}, and target detection \cite{Oreifej2013simultaneous,Sabushimike2016low} .

Solving Robust PCA as formulated in \eqref{eq:RPCA_problem} is an NP-hard problem in general.  Provable solutions for the problem were first provided in \cite{Candes2011robust, Chandrasekaran2009ranksparsity} by solving the convex relaxation of the problem
\begin{equation}
	\min_{L\in\pR^{m\times n}} \|L\|_* + \lambda \|S\|_1, \quad \text{s.t.}\quad M = L + S, \label{eq:RPCA_problem_convex}
\end{equation}
where $\|\cdot\|_*$ denotes the Schatten 1-norm\footnote{The Schatten 1-norm is often also referred to as the nuclear norm \cite{Recht2010guaranteed}.} of a matrix (sum of its singular values) and $\|\cdot\|_1$ denotes the $l_1$ norm of a vectorised matrix (sum of absolute values of its entries).  In \cite{Candes2011robust}, authors show that exact decomposition of a low-rank plus sparse matrix is possible for randomly chosen sparsity locations even for the case of the sparsity level $s$ being a fixed fraction $\alpha m n$ with $\alpha \in (0,1)$.  The work of \cite{Chandrasekaran2009ranksparsity} takes a deterministic approach in which corrupted entries can have arbitrary locations but must be sufficiently spread such that the sparsity fraction of each row and column does not exceed $\alpha$.  In both the works of \cite{Candes2011robust} and \cite{Chandrasekaran2009ranksparsity}, as well as subsequent extensions, it is common to impose conditions on the singular vectors of the low-rank component being sufficiently uncorrelated with the canonical basis.

Robust PCA is closely related to the problem of recovering a low-rank matrix from incomplete observations referred to as matrix completion \cite{Recht2010guaranteed}.  The main difference between the two is that, in the case of a matrix completion, the indices of missing entries are known, and the aim is to solve
\begin{equation}\label{eq:MC_problem}
\min_{L\in \pR^{m\times n}} \|\PO \left(L\right) - \PO \left(M\right) \|_F,\quad \text{s.t.}\quad L\in\LS_{m,n}(r,0), \,\, |\Omega^c| = s,
\end{equation}
where $\PO$ is entry-wise subsampling of observed entries of $M$ with indices in $\Omega$.

Similarly to the case of Robust PCA, matrix completion can be approached by solving a convex relaxation formulation of the problem \cite{Candes2009exact, Candes2010thepower, Recht2010guaranteed}, but there are also a number of algorithms that solve the non-convex formulation directly while also providing recovery guarantees \cite{Cai2010SVT, Haldar2009rank, Kyrillidis2014matrix, Lee2010admira, Tanner2013normalized,Tanner2014alternating,Wen2012solving}. Such non-convex methods are typically observed to be able to recover matrices with higher ranks than is possible by solving the convex relaxed problem \cite{Tanner2013normalized}.

\subsection{Main contribution}
It is well known that the model $\LS_{m,n}(r,s)$ from \eqref{eq:LS_def} need not have a unique solution without further constraints, such as the singular vectors of the low-rank component being uncorrelated with the canonical basis as quantified by the incoherence condition with parameter $\mu$
\begin{gather}
	\label{eq:coherence}
	\max_{i\in \left\{1,\ldots, r\right\}} \left\| U^* e_i \right\|_{2} \leq \sqrt{\frac{\mu r}{m}}, \qquad \max_{i\in \left\{1,\ldots, r\right\}} \left\| V^*e_i \right\|_{2} \leq \sqrt{\frac{\mu r}{n}},
\end{gather}
where $L = U \Sigma V^*$ is the singular value decomposition of the rank $r$ component $L$ of size $m\times n$. The incoherence condition for small values of $\mu$ ensures that left and right singular vectors are well spread out and not sparse \cite{Candes2009exact, Recht2010guaranteed}.

Trivial examples of matrices with non-unique decompositions in $\LS(r,s)$ include any matrix with two nonzero entries in differing rows and columns as they are in $\LS(r,s)$ for any $r$ and $s$ such that $r+s=2$ with the entries of the matrix assigned to the sparse or low-rank components selected arbitrarily.  Moreover, completion of a low-rank matrix is impossible for sampling patterns $\PO$ that are disjoint from the support of the matrix $M$, which can be likely for matrices that have few nonzeros.  Both of the aforementioned problems are overcome by imposing a low coherence which ensures the singular vectors of the low-rank matrix have most entries being nonzero \cite{Chandrasekaran2009ranksparsity}. 

Herein we highlight the presence of a more fundamental difficulty: There are matrices for which Robust PCA and matrix completion have no solution in that iterative algorithms that attempt to solve them can generate sequences of iterates $(L^t, S^t)$ for which $\lim_{t\to\infty} \| M - (L^t + S^t)\|_F = 0$ and $L^t + S^t \in \LS(r,s)$ for all $t$, but $M^* = \lim_{t\to\infty} L^t + S^t \not\in \LS(r,s)$.
This is not because of the ambiguity between possible solutions or lack of information about the matrix, but instead because $\LS_{m,n}(r,s)$ is not a closed set.
Moreover, this is not an isolated phenomenon, as sequences of $\LS_{m,n}(r,s)$ matrices converging outside of the set can be constructed for a wide range of ranks, sparsities and matrix sizes.



\begin{theorem}[$\LS_n(r,s)$ is not closed] \label{thm:main_intro}
The set of low-rank plus sparse matrices $\LS_n(r,s)$ is not closed for $r\geq 1$, $s\geq1$ provided $(r+1)(s+2) \leq n$, or provided $(r+2)^{3/2}s^{1/2} \leq n$ where $s$ is of the form $s = p^2 r$ for an integer $p\geq 1$.
\end{theorem}
\begin{proof}
By Theorem \ref{thm:main} and Theorem \ref{lemma6}.
\end{proof}

Theorem \ref{thm:main_intro} implies that there are matrices $M$ such that problem \eqref{eq:RPCA_problem} is ill-posed in that there are sequences $M^t=L^t+S^t$: for which $M^t\in\LS_n(r,s)$ for all $t$ but for which  $\lim_{t\rightarrow\infty} M^t=M\notin LS_n(r,s)$; moreover, the proof of Theorem \ref{thm:main} and Theorem \ref{lemma6} is constructive in that we design the matrices $L^t$ and $S^t$ to satisfy the aforementioned property.  The problem size bounds in Theorem \ref{thm:main_intro} allow for matrices with $r = \bigO(n^l)$ to have number of corruptions of order $s = \bigO(n^{2-3l})$ for $l\in[0,1/2]$, which for constant rank allows $s$ to be quadratic in $n$, and for $\l\in(1/2,1]$ to have the number of corruptions of order $s=\bigO(n^{(1-l)})$.  In Section \ref{sec:example} we illustrate the non-closedness of $\LS_3(1,1)$ and the consequent ill-posedness of the corresponding Robust PCA and low-rank matrix completion problems.

\subsubsection{Simple example of $\LS_3(1,1)$ not being closed}\label{sec:example}
Consider solving for the optimal $\LS(1,1)$ approximation to the following $3\times 3$ matrix, which is a special case of construction given in \cite{Kumar2014using} in the context of the matrix rigidity function not being lower semicontinuous.
\begin{equation}
\label{eq:RPCA_example}
\begin{gathered}
	\min_{X\in \pR^{3\times 3}} \|X-M\|_F, \qquad \text{s.t.}\quad X\in \LS(1,1),\\
	M = \begin{pmatrix}
        			0  & 1 & 1\\
        			1 & 0  & 0 \\
        			1 & 0  & 0
        		\end{pmatrix}
\end{gathered}
\end{equation}
Consider the following sequence of matrices $X_{\epsilon}$
\begin{gather*}
		X_\epsilon = \begin{pmatrix}
        				0  & 1 & 1\\
        				1 & \epsilon  & \epsilon \\
        				1 & \epsilon  & \epsilon
        		\end{pmatrix} \in \LS(1,1) \\ 
		= \underbrace{\begin{pmatrix}
        				1/\epsilon  & 1 & 1\\
        				1 & \epsilon  & \epsilon \\
        				1 & \epsilon  & \epsilon
        		\end{pmatrix}}_{L_\epsilon} + 
		\underbrace{\begin{pmatrix}
        				- 1/\epsilon  & 0 & 0\\
        				0 & 0  & 0 \\
        				0 & 0  & 0
        		\end{pmatrix}}_{S_\epsilon},
\end{gather*} 
which can decrease the objective function $\|X_\epsilon-M\|_{F} = 2\epsilon$ to zero as $\epsilon\rightarrow 0$, but at the cost of the constituents $L_{\epsilon}$ and $S_{\epsilon}$ diverging with unbounded energy. 
Moreover, the sequence which minimizes the error converges to a matrix $M$ lying outside of the feasible set $\LS(1,1)$ and is in the set $\LS(1,2)$ instead.
By the fact that $M\not\in\LS(1,1)$, we have that zero objective value cannot be attained and therefore one cannot construct sequences that yield the desired solution. Therefore Robust PCA as posed in \eqref{eq:RPCA_example} does not have a global minimum.
As the objective function is decreased towards zero, the energy of both the low-rank and the sparse components diverge to infinity.  Likewise, we could consider an instance of the matrix completion problem \eqref{eq:MC_problem} in which the top left entry of $M$ is missing and a rank $1$ approximation is sought. We see that a rank $1$ solution cannot be obtained as there does not exist a choice for the top left entry that would reduce the rank of $M$ to $1$. However, the sequence $L_\epsilon$ decreases the objective arbitrarily close to zero while the energy of the iterates grows without bounds, $\|L_\epsilon\|_F \rightarrow \infty$.

\subsection{Connection with matrix rigidity}
Robust PCA is closely related to the notion of the \textit{matrix rigidity} function which was originally introduced in complexity theory by Valiant \cite{Valiant1977graph} and refers to the minimum number of entries of $M$ that must be changed in order to reduce it to rank $r$ or lower.
\begin{equation*}
	\Rig(M, r) = \min_{S \in \pR^{m \times n}} \|S\|_0, \,\, \text{s.t.} \,\, \rank(M-S) \leq r.\footnote{Note that the original definition \cite{Valiant1977graph} works with $\rank(M+S)\leq r$. Here, we change the sign to be consistent with RPCA notation, $M = L+S$ and $\rank(L)\leq r$.}
\end{equation*}
Matrix rigidity is upper bounded for any $M\in\pR^{n\times n}$ and rank $r$ as 
\begin{equation}
	\Rig(M,r) \leq (n-r)^2.
\end{equation}
due to elementary matrix properties \cite{Valiant1977graph}.  
Matrices which achieve this upper bound for every $r$ are referred to as \textit{maximally rigid} and it was only recently showed in \cite{Kumar2014using} how to construct them explicitly, which was a long standing open question originally posed by Valiant in 1977. 

Matrix rigidity has important consequences for complexity of linear algebraic circuits but is also of interest for its mathematical properties. 
The work of \cite{Kumar2014using} also provides an example of the rigidity function not being lower semicontinuous, which implies the set $\LS_{m,n}(1,1)$ is not closed. 
Here, we generalize the result, providing non-closedness examples for many ranks, sparsities and matrix sizes, and discuss consequences for Robust PCA and matrix completion problems.  In Section \ref{sec:proofs} we prove Theorem \ref{thm:main_intro} and in Section \ref{sec:numerics} we illustrate how this phenomenon can cause several Robust PCA and matrix completion algorithms to diverge.


\section{Main result}\label{sec:proofs}

We extend the example of $\LS_{3}(1,1)$ with $M_3\in\pR^{3\times 3}$ given in \eqref{eq:RPCA_example} by constructing $\MO, \MT\not\in\LS_{n}(r,s)$ and yet for which there exists a sequence of matrices $M^{(i)}_n(\epsilon)$ which are in $\LS_{n}(r,s)$ and $\lim_{\epsilon \rightarrow 0} \|M_n^{(i)} - M^{(i)}_n(\epsilon)\|_F$ = 0. 
Matrix $\MOS$ as in \eqref{eq:lemma2_construction} demonstrates that $\LS_{n}(r,s)$ is not closed for $r \leq s$ (Lemma \ref{lemma2}) and matrix $\MTS$ as in \eqref{eq:lemma3_construction} is constructed for $r > s$ (Lemma \ref{lemma3}). 
In both cases we require $n$ to be sufficiently large in terms of $r$ and $s$.

For the case $r\leq s$, consider $\MO$ and $\MOS$ of the following general form
\begin{equation}\label{eq:lemma1_construction}
	\MO = \begin{pmatrix}
		0_{r,r} & A \\
        B & 0_{n-r,n-r} \\
	\end{pmatrix}, \qquad 
	\MOS = \begin{pmatrix}
		0_{r,r} & A \\
		B & \epsilon B\,A
	\end{pmatrix},
\end{equation}
where $A, B^T\in\pR^{r\times (n-r)}$ and $0_{k,k}$ denotes the $k\times k$ matrix with all zero entries. 
These constructed matrices satisfy the following properties. 
\begin{lemma}[General form of $\MO$]\label{lemma1}
Let $\MO$ and $\MOS$ be as defined in \eqref{eq:lemma1_construction}. Then $\MOS\in\LS(r,r)$. Furthermore 
\begin{equation}
\lim_{\epsilon \rightarrow 0} \| \MOS - \MO \|_F =  0. \label{eq:lemma1_convergence}
\end{equation}
\end{lemma}
\begin{proof}
We can write $\MOS$ in the form
\begin{equation}
\begin{pmatrix}
    \frac{1}{\epsilon} I_r\\
    B
    \end{pmatrix}
    \begin{pmatrix}
    I_r & \epsilon A
    \end{pmatrix} + 
    \begin{pmatrix}
    	-\frac{1}{\epsilon} I_r & 0\\
        0 & 0
    \end{pmatrix},
\end{equation}
which shows that $\MOS\in\LS_n(r,r)$. It also follows trivially from the definition \eqref{eq:lemma1_construction} that \eqref{eq:lemma1_convergence} is satisfied.
\end{proof}

\begin{remark}[Nested property of $\LS(r,s)$ sets]
Note that $\LS(r,s)$ sets form a partially ordered set
\begin{equation}
	\LS(r,s) \subseteq \LS(r',s'),
\end{equation}
for any $r'\geq r$ and $s'\geq s$.
As a consequence $\MOS\in\LS_n(r,r)$ implies that also $\MOS\in\LS_n(r,s)$ for $s\geq r$.
\end{remark}

With Lemma \ref{lemma1} we give the general form of $\MO$ and $\MOS$ such that $\MOS\in\LS_n(r,s)$ for $s\geq r$. 
It remains to show that, for a more specific choice of $A$ and $B$, we also have $\MO\not\in\LS_n(r,s)$. In particular, we construct $\MO$ and $\MOS$ as follows. 
\begin{equation}
\begin{gathered}\label{eq:lemma2_construction}
\MO = \begin{pmatrix}
	0_{r,r} & \beta & A^{(1)} & \ldots & A^{(l)} \\
    \alpha^T & 0_{k,k}  & \ldots & \ldots & 0_{k,r}\\
    B^{(1)} & \vdots & \ddots & & \vdots\\
    \vdots & \vdots & & \ddots & \vdots \\
    B^{(l)} & 0_{r,k} & \ldots & \ldots & 0_{r,r} 
    \end{pmatrix}, \vspace{0.2em}\\
	\MOS = \begin{pmatrix}
		0_{r,r} & \beta & A^{(1)} & \ldots & A^{(l)} \\
   		\alpha^T & \epsilon  \alpha^T \beta  & \ldots & \ldots & \epsilon\alpha^TA^{(l)} \\
    		B^{(1)} & \vdots & \ddots & & \vdots\\
    		\vdots & \vdots & & \ddots & \vdots \\
   		B^{(l)} & \epsilon B^{(l)}\beta & \ldots & \ldots & \epsilon B^{(l)}A^{(l)}\\
	\end{pmatrix},
\end{gathered}
\end{equation}
where $\alpha,\beta\in\pR^{r \times k}$ are matrices with all non-zero entries, $A^{(i)}, B^{(i)}\in \pR^{r\times r}$ are arbitrary non-singular matrices which may, but need not, be the same, $0_{a,b}$ and $\mathbbm{1}_{a,b}$ denote $a\times b$ matrices with all entries equal to zero or one respectively, and we set $l = \ceil{(s+1)/2}$, $k = \ceil{l/r}$.

By construction, the matrix size is $n = r(l+1) + k$, due to the $l$
matrices $A^{(i)}$ and $B^{(i)}$ for $i=1,\ldots,l$ each being of size $r\times r$, the top left $r\times r$ zero matrix and $k$ columns of $\alpha$ and $\beta$.

\begin{lemma}
\label{lemma2}
$\LS_{n}(r,s)$ is not closed for $1\leq r \leq s$ provided 
\begin{equation}\label{eq:lemma2_size}
	n\geq r\left(\left\lceil \frac{s+1}{2} \right\rceil + 1\right) + \left\lceil \frac{\ceil{(s+1)/2}}{r} \right\rceil.
\end{equation} 
\end{lemma}
\begin{proof} 
Take $\MO$ as in \eqref{eq:lemma2_construction}. 
By Lemma \ref{lemma1} there exists a matrix sequence $\MOS\in\LS_{n}(r,r)$ such that $\| \MOS - \MO \|_F \rightarrow 0$ as $\epsilon \rightarrow 0$. 
Since for $r\leq s$ we have $\LS_n(r,r)\subseteq \LS_n(r,s)$, it follows also that $\MOS \in \LS_n(r,s)$. 

It remains to prove that $\MO\not\in\LS_n(r,s)$, which is equivalent to showing $\Rig(\MO, r) > s$. 
We show that having a sparse component $\|S\|_0 \leq s$ is insufficient for $\rank(\MO-S)\leq r$, because for any choice of such $S$ with at most $s$ non-zero entries, the matrix $\MO - S$ must have a $(r+1)\times(r+1)$ minor with nonzero determinant implying $\rank(\MO-S)\geq r+1$.

In order to establish $\rank(\MO-S)\geq r+1$ we consider $2l$ minors of $\MO$ each of size $(r+1)\times (r+1)$. 
For $l$ of these we select minors that include $A^{(i)}$, $i = 1, \ldots, l$, along with an additional column from the first $r$ columns and an additional row entry from row index $r+1$ to $k+r$ from $\MO$; and for the remaining $l$ minors we similarly choose a $B^{(i)}$ and an additional row and column as before.

These minors $C_i$ are constructed as \eqref{eq:lemma2_minors}
\begin{align}
	C_i = \begin{cases}
		\begin{pmatrix}
		 0_{r,1} & A^{(i)}  \\
		  \alpha_i & 0_{1,r} \\
	\end{pmatrix}, \quad i = 1, \,\ldots, \,l, \vspace{0.5em} \\
	\begin{pmatrix}
		0_{1,r} &  \beta_{i-l}    \\
		B^{(i-l)} & 0_{r,1} \\
	\end{pmatrix}, \quad i = l+1, \,\ldots, \,2l,
	\end{cases} \label{eq:lemma2_minors}
\end{align}
where $0_{u,v}$ denotes the $u\times v$ matrix with all entries equal
to zero, $\alpha_i,\beta_i\ne 0$ are chosen to be different entries
from $\alpha, \beta\in\pR^{r\times k}$ for each $i = 1, \ldots, l$
with $k=\lceil l/r\rceil$, and $A^{(i)}, B^{(i)}$ are each full rank.
Note that matrices $C_i$ do not have disjoint supports as they share the left $r$ zero entries in the first row of $C_i$ for $i = 1,\,\ldots,\, l$ and the top $r$ zero entries in the first column of $C_i$ for $i = (l+1,\,\ldots,\, 2l)$. 
We refer to these entries as the {\it intersecting part} of $C_i$.

The $S$ such that $\rank(\MO - S) = r$ must have at least $2l$
nonzeros, thus $\Rig(\MO, r)\geq 2l$, by noting that although the $C_i$ have intersecting portions, $S$ restricted to the $i^{th}$ subminor associated with $C_i$ will have at least one distinct nonzero per $i$. 
Consider the $C_i$ for $i = 1, \ldots, l$ associated with $\alpha_i$ and $A^{(i)}$ and let $S_i$ be the corresponding $(r+1)\times (r+1)$ sparsity mask of $S$. 
It follows that $S_i$ must have at least one entry in the non-intersecting set otherwise $C_i + S_i$ is of the form
\begin{equation}
		C_i+S_i = \begin{vmatrix}
		 \vline & &  &  \\
		 s_i & &  A^{(i)}&   \\
		\vline & &  &   \\
		 \alpha_i &  0 & \ldots & 0 \,
	\end{vmatrix} = \alpha_i |A^{(i)}| \neq 0,		\label{eq:lemma2_intersect}
\end{equation}
which is insufficient for the rank of $C_i$ to become rank deficient; similarly for $i = l+1, \ldots, 2l$.

With $\Rig(\MO, r) \geq 2l$ we set $l = \ceil{(s+1)/2}$, which then
implies that $\MO \not\in\LS_n(r,s)$ and by the construction of $\MO$ 
\begin{equation} \label{eq:lemma2_size_old}
	 n \geq r(l + 1) + k.
\end{equation}
Substituting $l = \ceil{(s+1)/2}$ and $k = \ceil{l/r}$, we conclude that $\LS_n(r,s)$ is not a closed set for $s \geq r \geq 1$ provided 
\begin{equation}
n \geq r\left(\left\lceil \frac{s+1}{2} \right\rceil + 1\right) + \left\lceil \frac{\ceil{(s+1)/2}}{r} \right\rceil.
\end{equation}
\vspace{0.2em}
\end{proof}

Turning to the $r>s$ case, we now build upon Lemma \ref{lemma3} by constructing matrices $\MT$ and $\MTS$ as
\begin{equation}\label{eq:lemma3_construction}
	\begin{gathered}
	\MT = \begin{pmatrix}
    	\MTH & 0 & \ldots & 0 \\
    	0 & \MTN^{(1,1)} & \ldots & \MTN^{(1,s+1)}  \\ 
    	\vdots & \vdots & \ddots & \\
    	0 & \MTN^{(s+1,1)} &  & \MTN^{(s+1,s+1)} 
    \end{pmatrix} = \begin{pmatrix}
    	\MTH& 0_{n',\,(s+1)(r-s)}  \\
    	0_{(s+1)(r-s),\,n'} & \MTN
    \end{pmatrix} \vspace{0.2em} \\
    \MTS =  \begin{pmatrix}
    	\MTHS & 0_{n',\,(s+1)(r-s)}  \\
    	0_{(s+1)(r-s),\,n'} & \MTN
    \end{pmatrix}
    \end{gathered},
\end{equation}
where $\MTN^{(i,j)}\in\pR^{(r-s)\times (r-s)}$ are identical full rank matrices and  
\begin{equation}\label{eq:lemma3_construction2}
	\MTH= \begin{pmatrix}
	0_{s,s} & \beta & A^{(1)} & \ldots & A^{(l)} \\
    	\alpha^T & 0  & \ldots & \ldots & 0_{1,s}\\
    	B^{(1)} & \vdots & \ddots & & \vdots\\
    	\vdots & \vdots & & \ddots & \vdots \\
    	B^{(l)} & 0_{s,1} & \ldots & \ldots & 0_{s,s} \\
\end{pmatrix}, \quad
\MTHS =  \begin{pmatrix}
	0_{s,s} & \beta & A^{(1)} & \ldots & A^{(l)} \\
   	 \alpha^T & \epsilon  \alpha^T \beta  & \ldots & \ldots & \epsilon\alpha^TA^{(l)} \\
   	 B^{(1)} & \vdots & \ddots & & \vdots\\
  	  \vdots & \vdots & & \ddots & \vdots \\
    	B^{(l)} & \epsilon B^{(l)}\beta & \ldots & \ldots & \epsilon B^{(l)}A^{(l)}	\end{pmatrix},
\end{equation}
have the same structure as in \eqref{eq:lemma2_construction} but with $r$ replaced by $s$ and as a result $A^{(i,j)}, B^{(i,j)}\in\pR^{s\times s}$, $\alpha, \beta \in\pR^s$, $l = \ceil{(s+1)/2}$, so $\MTH\not\in\LS_{n'}(s,s)$ while $\MTHS\in\LS_{n'}(s,s)$.

By construction, the size of $\MTH$ is $n' = s(l+1) + 1$ and the size of $\MT$ is $n= n' + (s+1)(r-s)$.

\begin{lemma}
\label{lemma3} 
$\LS_{n}(r,s)$ is not closed for $r>s\geq 1$ provided 
\begin{equation} \label{eq:lemma3_size}
n \geq s\left(\left\lceil \frac{s+1}{2} \right\rceil + 1\right)+ 1 +(s+1)(r-s).
\end{equation}
\end{lemma}
\begin{proof}
Consider $\MT$ and $\MTS$ from \eqref{eq:lemma3_construction}. 
By additivity of rank for block diagonal matrices, $\rank\left(\MTN \right) = (r-s)$ and $\MTHS\in\LS_{n'}(s,s)$, we have that $\MTS\in\LS_{n}(r, s)$.

That $\MT\not\in\LS_n(r, s)$ follows from $\Rig(\MT, r) > s$ which 
we show via $\| S\|_0\leq s$ being insufficient for $\rank(\MT - S) \leq r$, because for any such $S$, matrix $(\MT - S)$ must have at least one $(r + 1)\times (r +1)$ minor with non-zero determinant, implying $\rank(\MT - S) \geq r +1$.

We consider minors $D_i$ of size $(r + 1)\times (r +1)$ by diagonally appending a minor $\hat{C}_i\in\pR^{(s+1) \times (s+1)}$ of $\MTH$ of a similar structure as in \eqref{eq:lemma2_minors} and the whole $i^{th}$ diagonal block $\MTN^{(i,i)}\in\pR^{(r-s)\times (r-s)}$ 
\begin{equation}
\begin{gathered}
	D_i = \begin{pmatrix}
    		\hat{C}_i & 0 \\
        		0 & \MTN^{(i,i)}
    	\end{pmatrix}, \qquad i = 1, \,\ldots, \,s+1.
\end{gathered}
\end{equation}

The intersecting supports between $D_i$ come from the
  intersecting parts between individual $\hat{C}_i$ as explained in
  \eqref{eq:lemma2_minors} in the proof of Lemma \ref{lemma2} due to
  matrices $\MTN^{(i,i)}$ being selected from the block diagonal.
In order that $\rank{(D_i)}\leq r$ requires $S_i$ to have at least one non-zero in a part of $D_i$ that is disjoint from $D_j$ for $j\neq i$. 
Either $S_i$ has at least one non-zero on a zero block or $\MTN^{(i,j)}$ or $\hat{C}_i$. 
If the non-zero is in a zero block or $\MTN^{(i,j)}$, then these are disjoint which implies at least $s+1$ non-zero entries. 
On the other hand, if the non-zero is in $\hat{C}^{(i)}$ then at least one entry of $\MTN$ must be changed in the non-intersecting part of $\hat{C}_i$ as argued following equation \eqref{eq:lemma2_minors}. 
Therefore for every $D_i$ at least one distinct entry per $i$ must be
changed using the corresponding sparsity component $S_i$, and since $i
= 1, \ldots, s+1$, we must also change at least $s+1$ entries of
$\MT$; that is $\Rig(\MT, r)\geq s+1$.

By the construction of $\MT$ in this argument we have
\begin{equation}
n \geq \underbrace{s(l+1)+ 1}_{n'\text{, size of } \MTH} +\underbrace{(s+1)(r-s)}_{\text{size of }\mathbbm{1}_{s+1} \otimes N},	\label{eq:lemma3-size}
\end{equation}
where the size of $\MTH$ comes from $l$ times repeating the matrices $A^{(i)}$ and $B^{(i)}$ each of size $s\times s$, the top left $s\times s$ matrix $0_{s,s}$, the $\beta$ column and $\alpha$ row respectively and $s+1$ times repeating matrix $\MTN$ of size $(r-s)$. By zero padding of the matrix we can arbitrarily increase its size.
Substituting $l= \ceil{(s+1)/2}$ gives that $\LS_n(r,s)$ is not a closed set for $r > s$ provided 
\begin{equation}
n \geq s\left(\left\lceil \frac{s+1}{2} \right\rceil + 1\right)+ 1 +(s+1)(r-s).
\end{equation}
\end{proof}

The following theorem gives a sufficient lower bound on the matrix size such that both size requirements derived in Lemma \ref{lemma2} and Lemma \ref{lemma3} are met, thus unifying both results.
\begin{theorem} \label{thm:main}
The low-rank plus sparse set $\LS_n(r,s)$ is not closed provided $n\ge
(r+1)(s+2)$ and $r\geq 1$, $s\geq1$.
\end{theorem}
\begin{proof}
Suppose $n \geq (r+1)(s+2)$. We show that this is a sufficient condition for the matrix size requirements in \eqref{eq:lemma2_size} in Lemma \ref{lemma2} and \eqref{eq:lemma3_size} in Lemma \ref{lemma3} to hold.

We first obtain a sufficient condition on the matrix size in \eqref{eq:lemma2_size} in Lemma \ref{lemma2}, bounding
\begin{align}
	&\,r \left( \left\lceil \frac{s+1}{2} + 1 \right\rceil  \right) +  \left\lceil \frac{\ceil{(s+1)/2}}{r} \right\rceil \nonumber\\
	\leq &\,r\left(\frac{s+1}{2}+2\right) +  \left(\frac{1}{r}\right)\left(\frac{s+1}{2}+1\right) +1 \nonumber\\
	\leq&\,r \left( \frac{s+5}{2} \right) + \left(\frac{s+5}{2}\right) = (r+1)\left(\frac{s+5}{2}\right) \nonumber\\
	\leq&\,(r+1)(s+2),	\label{eq:lemma2_inequality}
\end{align}
where the first inequality in \eqref{eq:lemma2_inequality} comes from an upper bound on the ceiling function $\ceil{x} \leq x+1$, the second inequality follows from $r\geq 1$ and the last inequality holds for $s\geq 1$.

We also obtain a sufficient bound condition on the matrix size in \eqref{eq:lemma3_size} in Lemma \ref{lemma3} of the form
\begin{align}
	&\,s\left( \left\lceil \frac{s+1}{2} + 1 \right\rceil  \right) + 1 + (s+1)(r-s)\nonumber\\
	\leq &\,s\left(\frac{s+1}{2}+2\right) + (s+1)(r-s) = -\frac{s^2}{2} + \frac{3}{2} + rs +1 \nonumber\\
	\leq&\,(r+1)(s+1) \leq (r+1)(s+2). \label{eq:lemma3_inequality}
\end{align}
The first inequality in \eqref{eq:lemma3_inequality} comes from an upper bound on the ceiling function and the second inequality holds for $s\geq1$.

Combining \eqref{eq:lemma2_inequality}, \eqref{eq:lemma3_inequality} with Lemma \ref{lemma2} and Lemma \ref{lemma3} gives that $\LS_{n}(r,s)$ is not a closed set for $n\geq (r+1)(s+2)$ and $r\geq1$, $s\geq1$.
\end{proof}

\subsection{Quadratic sparsity}

Note that the condition $n \geq (r+1)(s+1)$ limits the order of $r$ and $s$; in particular if $r = \bigO(n^l)$ then $s = \bigO(n^{1-l})$ which for $l \geq 0$ constrains $s$ to be at most linear in $n$, $s=\bigO(n)$.  In Lemma \ref{lemma4} and Lemma \ref{lemma5}, we extend the result so that for $r = \bigO(n^l)$ and $l\leq1/2$ we obtain $s = \bigO(n^{2-3l})$ which for constant rank, $l = 0$, allows $s$ to be quadratic $\bigO(n^2)$.

Lemma \ref{lemma4} establishes a lower bound on the rigidity of block matrices in terms of the rigidity of a single block. Lemma \ref{lemma5} shows that the sequence $K(\epsilon)$ converging to $K$ is an example of $\LS_{n}(r,p^2r)$ not being closed provided $n \geq p\left( r \left( \left\lceil \frac{r+1}{2}\right\rceil + 1\right) + 1\right) $. Let
\begin{equation}\label{eq:lemma4_construction}
	K = \begin{pmatrix}
		\MTH^{(1,1)}&  \cdots & \MTH^{(1,p)}   \\
		\vdots  & \ddots & \vdots \\
		\MTH^{(p,1)}& \cdots & \MTH^{(p,p)}
	\end{pmatrix}, \quad 
	K(\epsilon) = \begin{pmatrix}
		\hat{M}^{(1,1)}_{n'}(\epsilon) &  \cdots & \hat{M}^{(1,p)}_{n'}(\epsilon)    \\
		\vdots  & \ddots & \vdots \\
		\hat{M}^{(p,1)}_{n'}(\epsilon) & \cdots & \hat{M}^{(p,p)}_{n'}(\epsilon)
	\end{pmatrix}
\end{equation}
where matrices $\hat{M}^{(i,j)}_{n'}(\epsilon)\in\LS_{n'}(r,r)$ and $\MTH^{(i,j)}\not\in\LS_{n'}(r,r)$ are of the same structure as in \eqref{eq:lemma3_construction2} and $\lim_{\epsilon\rightarrow 0} K(\epsilon) = K$ where $K\in\pR^{(n'p)\times (n'p)}$ is constructed by repeating $\MTH$ in $p$ row and column blocks.

\begin{lemma}
\label{lemma4}
For $K$ as in \eqref{eq:lemma4_construction}
\begin{equation}
	\Rig(K,r) \geq p^2 \Rig(\MTH, r).
\end{equation}
\end{lemma}
\begin{proof}
Let $S$ be the sparsity matrix corresponding to $\Rig(K,r)$, such that
\begin{equation}
\begin{gathered}
\rank(K - S) \leq r, \qquad \|S\|_0 = \Rig(K,r), \\
\text{and} \qquad S = \begin{pmatrix}
		\hat{S}^{(1,1)}&  \cdots & \hat{S}^{(1,p)}   \\
		\vdots  & \ddots & \vdots \\
		\hat{S}^{(p,1)}& \cdots & \hat{S}^{(p,p)}
	\end{pmatrix},
\end{gathered}
\end{equation}
where $\hat{S}^{(i,j)}\in\pR^{n' \times n'}$ denotes the sparsity matrix used in the place of the $\MTH^{(i,j)}$ block.
A necessary condition for $\rank(K - S) \leq r$ is that also the rank of individual blocks is less than or equal to $r$, that is
\begin{equation}
	\rank(\MTH - \hat{S}^{(i,j)}) \leq r, \qquad \forall i,j\in\left\{1, \ldots, p \right\}.
\end{equation}
By definition of the rigidity function as the minimal sparsity of $S$ such that $\rank(\MTH - S)\leq r$, we have that
\begin{equation}
	\|\hat{S}^{(i,j)}\|_0 \geq \Rig(\MTH, r).
\end{equation}
Summing over all blocks $i,j\in\left\{1, \ldots, p \right\}$ yields the result
\begin{equation}
	\|S\|_0 = \sum_{i,j}^{p,p} \|\hat{S}^{(i,j)}\|_0 \geq \sum_{i,j}^{p,p}\Rig(\MTH, r),
\end{equation}
and consequently that
\begin{equation}
	\Rig(K, r) \geq p^2 \Rig(\MTH, r).	
\end{equation}
\end{proof}

\begin{lemma}\label{lemma5}
The low-rank plus sparse set $\LS_n(r,p^2 r)$ is not closed provided
\begin{equation*}
n \geq p\left( r \left( \left\lceil \frac{r+1}{2}\right\rceil + 1\right) + 1\right)
\end{equation*}
and $r\geq 1$, $p\geq1$.
\end{lemma}

\begin{proof}
Consider $K$ and $K(\epsilon)$ as in \eqref{eq:lemma4_construction}. 
Repeating $\MTH\in\LS_{n'}(r,r)$ $p$ times in row and column blocks does not increase the rank, so $\rank\left(K(\epsilon)\right) = r$ and by additivity of sparsity we have that $K\left(\epsilon\right)\in\LS_{n}(r, p^2r)$. 
By Lemma \ref{lemma4} and $\Rig(\MTH,r) > r$ we have the strict lower bound on the rigidity of $K$ 
\begin{equation}
	\Rig(K, r) \geq p^2 \Rig(\MTH,r) > p^2 r,
\end{equation}
which implies that $K\not\in\LS_n(r, p^2 r)$ while $K(\epsilon)\in\LS_n(r, p^2r)$.

Recall that the size  of $\MTH$ as defined in \eqref{eq:lemma3_construction2} is $n' = r(l+1) + 1$ and, since $\MTH$ is repeated $p$ times, we obtain
\begin{align}\label{eq:lemma5_size}
	n &\geq p\left(r(l+1) + 1\right) = p\left( r \left( \left\lceil \frac{r+1}{2}\right\rceil + 1\right) + 1\right),
\end{align}
where the inequality comes from zero padding of the matrix to arbitrarily expand its size.
\end{proof}

\begin{theorem}\label{lemma6}
The low-rank plus sparse set $\LS_n(r, s)$ is not closed provided
\begin{equation*}
n \geq (r+2)^{3/2}s^{1/2}
\end{equation*}
and $r\geq 1$, and $s$ is of the form $s = p^2 r$ for an integer $p\geq 1$.
\end{theorem}
\begin{proof}
We weaken the condition of Lemma \ref{lemma5} and show that it suffices to have $n\geq (r+2)^{3/2}s^{1/2}$ for $\LS_n(r, s)$ not closed by substituting $s = p^2r$
\begin{align}
 	&\,p\left( r \left( \left\lceil \frac{r+1}{2}\right\rceil + 1\right) + 1\right) = \left(\frac{s}{r}\right)^{\frac{1}{2}}\left( r \left( \left\lceil \frac{r+1}{2}\right\rceil + 1\right) + 1\right) \\
 	 \leq&\, s^{1/2} \left(r^{1/2} \left( \frac{r+5}{2}\right) + 1\right) = s^{1/2} \left( \frac{r^{3/2}}{2} + 2r^{1/2} + r^{-1/2} \right) \\
 	 \leq&\, s^{1/2} \left( \frac{r^{3/2}}{2} + 2r^{1/2} + \frac{3}{2}r^{-1/2} \right) =  s^{1/2} \frac{(r+1)(r+2)}{2 \sqrt{r}} \\
 	 \leq&\, s^{1/2}(r + 2)^{3/2},
 \end{align}
where in the first line we substitute $s = p^2r$, the first inequality comes from an upper bound on the ceiling function, the second inequality follows from $r^{-1/2}\leq\frac{3}{2}r^{-1/2}$, and the last inequality holds for $r \geq 1$.
\end{proof}

\subsection{Almost maximally rigid examples of non-closedness}
It remains to prove non-closedness of $\LS_n(r,s)$ sets for as high
ranks $r$ and sparsities $s$ as possible; partial results in this
direction are discussed in this section.
There cannot be a maximally rigid sequence converging outside $\LS\left(r,(n-r)^2\right)$ because $\LS\left(r,(n-r)^2\right)$ corresponds to the set of all $\pR^{n\times n}$ matrices. 
Similarly, it is necessary that both $r\geq 1$ and $s \geq 1$ hold since sets of rank $r$ matrices $\LS(r,0)$ and sets of sparsity $s$ matrices $\LS(0,s)$ are both closed.
As a consequence, the highest possible rank and sparsity for which $\LS_n(r,s)$ is not closed corresponds to one strictly less than the maximal rigidity bound, i.e. $\LS\left(r, (n-r)^2-1\right)$ for $r \geq 1$ and also $s = (n-r)^2-1 \geq 1$. 

It is shown in \cite{Kumar2014using} that the matrix rigidity function might not be semicontinuous even for maximally rigid matrices. This translates into the set $\LS_3(1,3)$ not being closed as we have $M(\epsilon)\in\LS_3(1, 3)$ which converges to $M\not\in\LS_3(1,3)$ by choosing
\begin{equation}\label{eq:notassigned}
	M = \begin{pmatrix}
    		a & b & c \\
		d & e & 0 \\
		g & 0 & i
    	\end{pmatrix} \quad \textrm{and} \quad
	M(\epsilon) = \begin{pmatrix}
		a & b & c \\
		d & e & \epsilon cd \\
		g & \epsilon bg & i
	\end{pmatrix}.
\end{equation}
It is easy to check that for a general choice of $\left\{a, \ldots, i\right\}$, $M$ is maximally rigid with $\Rig(M, 1) = 4$. However, $\Rig\left(M(\epsilon), 1\right) = 3$ since $M(\epsilon)$ can be expressed in the following way
\begin{equation}
	M(\epsilon) = \begin{pmatrix}
    		\epsilon^{-1} & b & c \\
		d & \epsilon bd & \epsilon cd \\
		g & \epsilon bg & \epsilon cg
    	\end{pmatrix} + \begin{pmatrix}
    		a-\epsilon^{-1} & 0 & 0 \\
		0 & e-\epsilon bd & 0 \\
		0 & 0 & i-\epsilon cg
    	\end{pmatrix}.
\end{equation}
Having established $\LS_3(1,3)$ is not a closed set, which is the optimal result with the highest possible sparsity for sets of rank $1$ matrices of size $3\times 3$.
We pose the question as to whether this result can be generalized and the following conjecture holds.
\begin{conjecture}[Almost maximally rigid non-closedness] \label{conj:maxrigid}
The low-rank plus sparse set $\LS_n(r,s)$ is not closed provided
\begin{equation}
	n \geq r + (s+1)^{1/2},
\end{equation}
for $s \in [1, (n-1)^2 -1]$ and $r \in [1, n-2]$. 
\end{conjecture}
%
%
%
%

\section{Numerical examples with divergent Robust PCA and matrix completion\label{sec:numerics}}

Theorem \ref{thm:main_intro} and the constructions in Section \ref{sec:proofs} indicate that there are matrices for which Robust PCA and matrix completion, as stated in \eqref{eq:RPCA_problem} and \eqref{eq:MC_problem} respectively, are not well defined.  In particular, the objective can be driven to zero while the components diverge with unbounded norms.  Herein we give examples of two simple matrices which are of a similar construction to $M$ in \eqref{eq:RPCA_example},
\begin{equation*}
	M^{(1)} = \begin{pmatrix}
		2 & -1 & -1\\
		-1 & 0 & 0 \\
		-1 & 0 & 0
	\end{pmatrix}, \quad 
	M^{(2)} = \begin{pmatrix}
		1 & -2 & -2\\
		-2 & 0 & 0 \\
		-2 & 0 & 0
	\end{pmatrix},
\end{equation*}
which are not in $\LS(1,1)$, but can be approximated by an arbitrarily close $M^{(1)}_\epsilon,M^{(2)}_\epsilon\in \LS(1,1)$, and for which popular RPCA and MC algorithms exhibit this divergence.  This is analogous to the problem of diverging components for CP-rank decomposition of higher order tensors which is especially pronounced for algorithms employing alternating search between individual components \cite{Silva2008tensor}.

\begin{figure}[t]
	\centering
	\begin{subfigure}{0.45\textwidth}
		\includegraphics[width=1\textwidth]{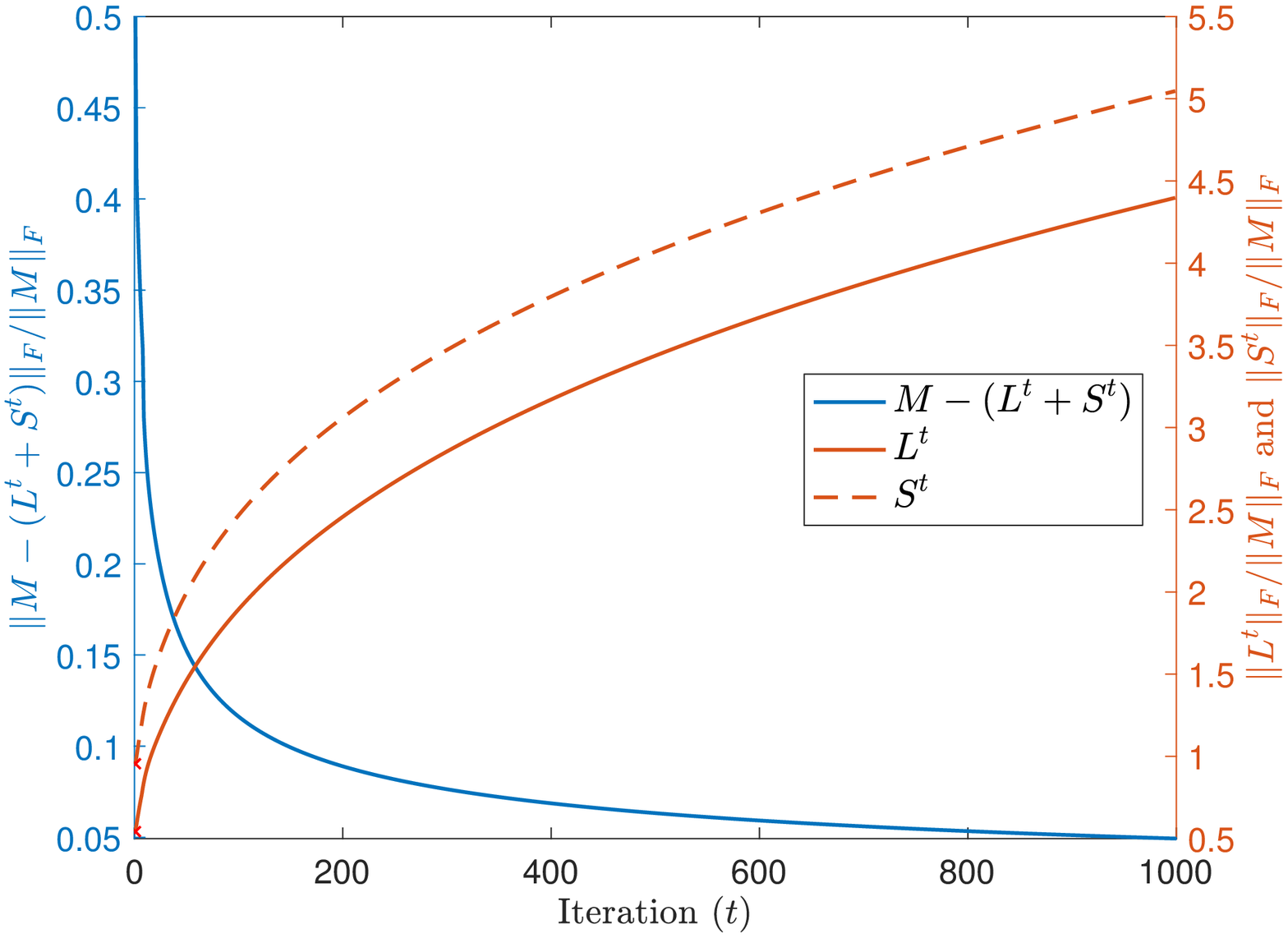}
		\caption{FastGD \cite{Yi2016fast} applied to $M^{(1)}$.\label{fig:rpca_fastgd1}}
	\end{subfigure}
	\hspace{0.05\textwidth}
	\begin{subfigure}{0.45\textwidth}
		\includegraphics[width=1\textwidth]{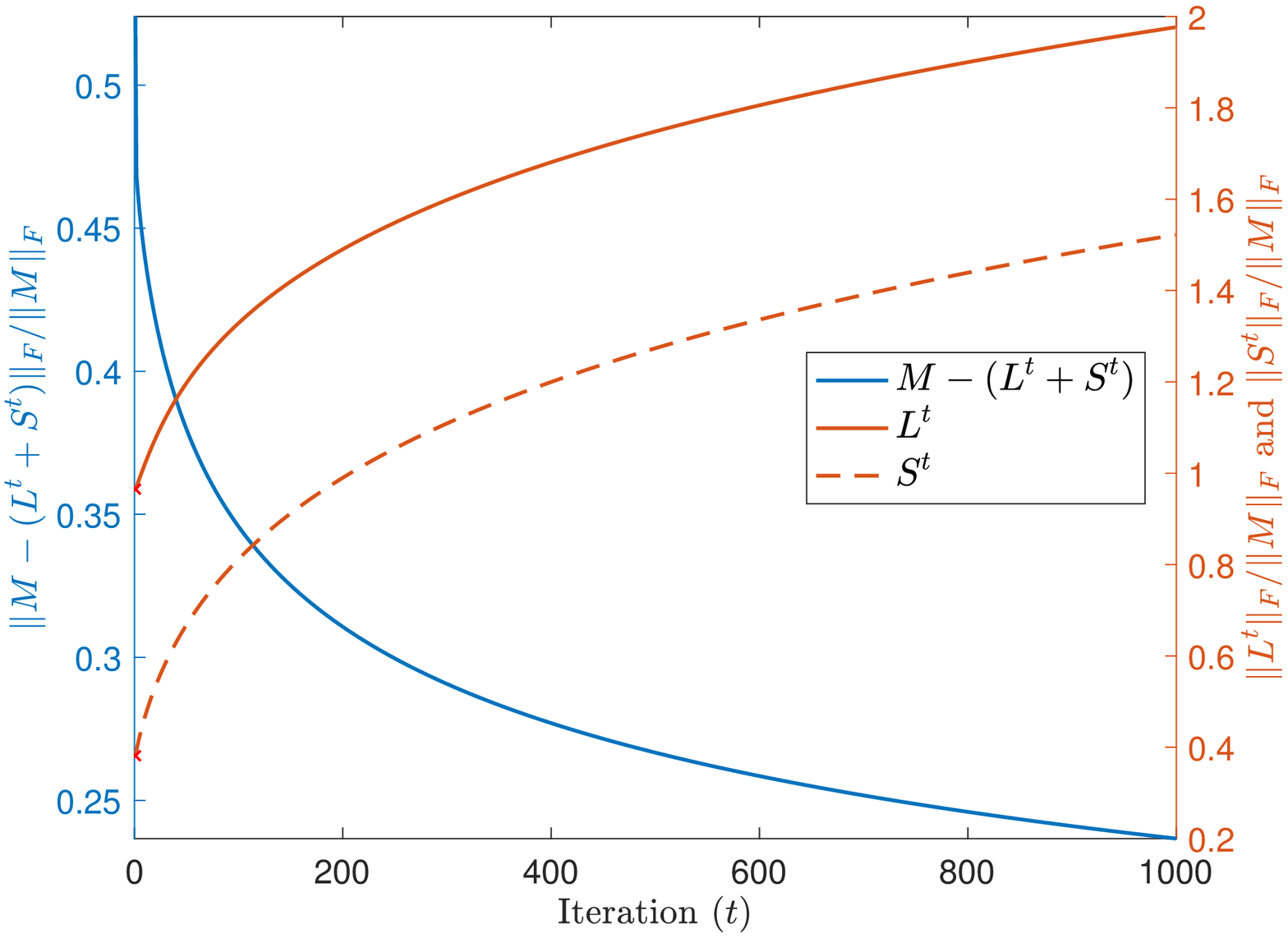}
		\caption{AltMin \cite{Gu2016low} applied to $M^{(2)}$\label{fig:rpca_altmin1}.}
	\end{subfigure}
	\begin{subfigure}{0.45\textwidth}
		\includegraphics[width=1\textwidth]{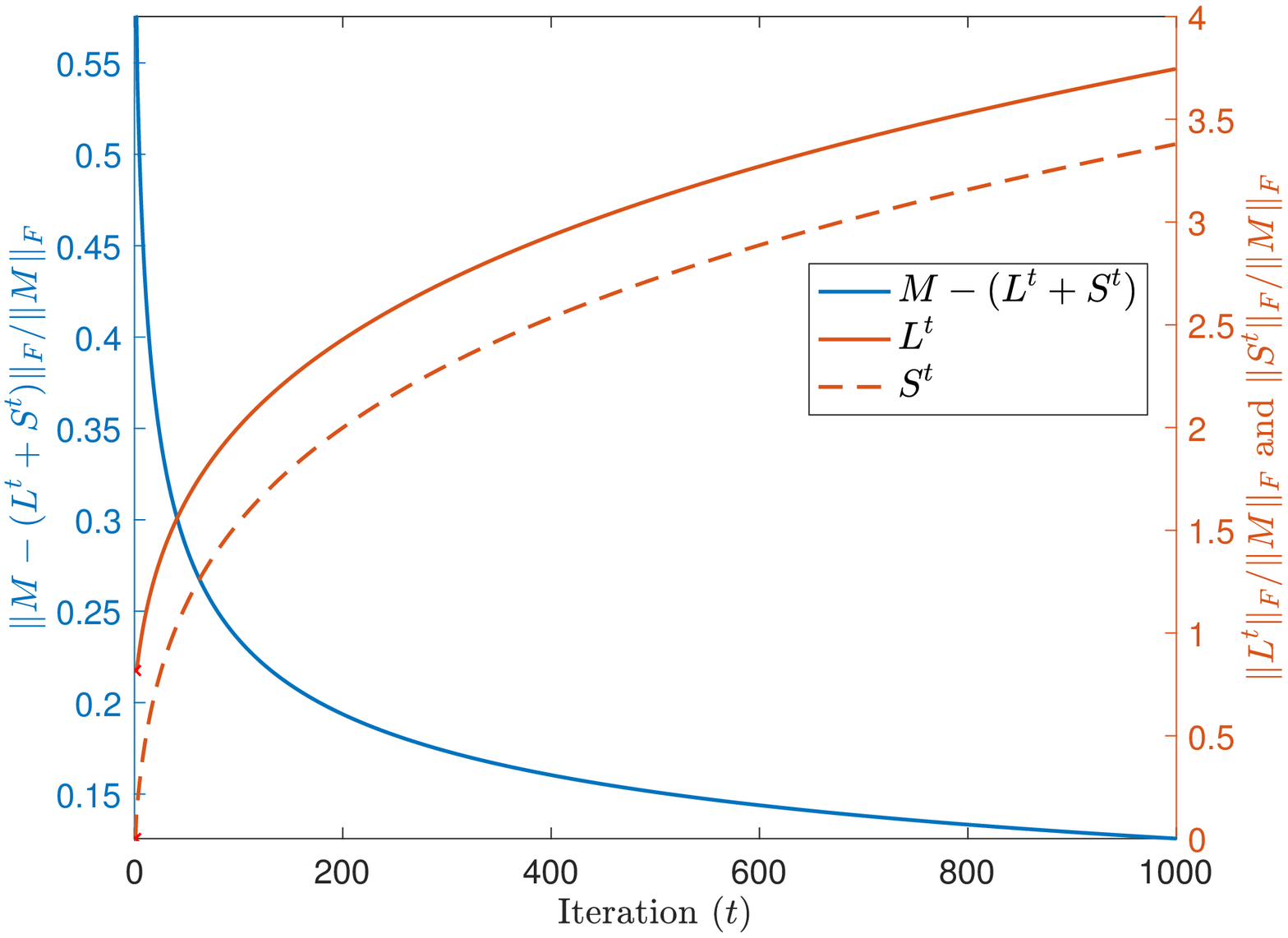}
		\caption{AltProj \cite{Dutta2018nonconvex} applied to $M^{(2)}$.\label{fig:rpca_altproj_pr1}}
	\end{subfigure}
	\hspace{0.05\textwidth}
	\begin{subfigure}{0.45\textwidth}
		\includegraphics[width=1\textwidth]{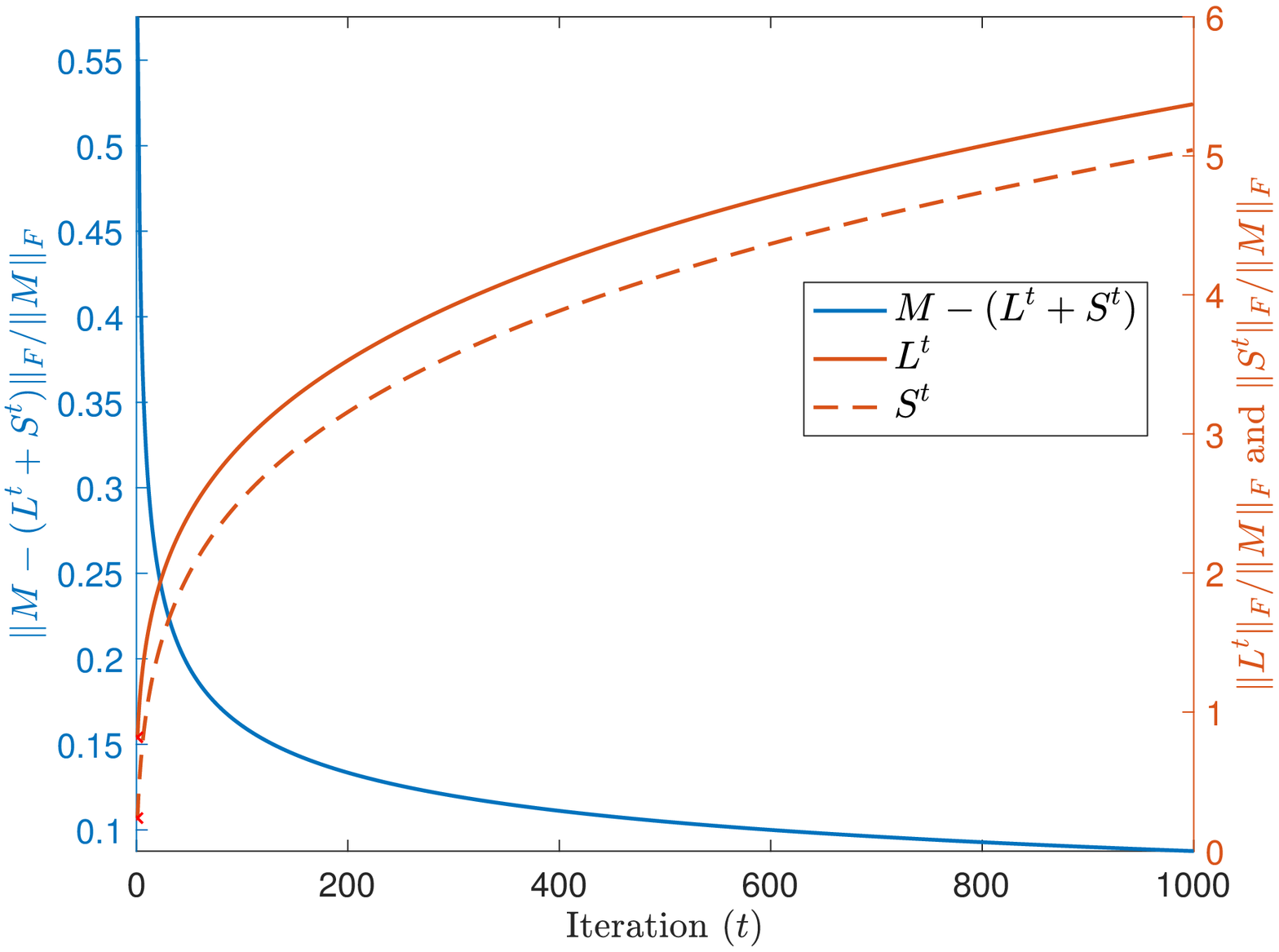}
		\caption{GoDec \cite{Zhou2011godec} applied to $M^{(2)}$.\label{fig:rpca_godec1}}
	\end{subfigure}
	\vspace{-1em}
	\caption{Solving for an $\LS(1,1)$ approximation to $M^{(1)}$ and $M^{(2)}$ using four non-convex Robust PCA algorithms. Despite the norm of the residual $\| M - (L^t +S^t)\|_F$ converging to zero, norms of the constituents $L^t, S^t$ diverge. We set algorithms parameters $r = 1, s = 1$ where possible. For FastGD we set $\lambda = 3.23$ and stepsize $\eta = 1/6$ which corresponds to choosing $s=1$. For GoDec we set the low-rank projection precision parameter to be $10$.\label{fig:rpca11-nonconvex}}
\end{figure}

Non-convex algorithms for solving the Robust PCA problem \eqref{eq:RPCA_problem} are typically observed to be faster than convex relaxations of the problem and often are able to recover matrices with higher ranks than possible by solving the convex relaxation \eqref{eq:RPCA_problem_convex}.  We explore the performance of four widely considered non-convex Robust PCA algorithms: Fast Robust PCA via Gradient Descent (FastGD) \cite{Yi2016fast}, Alternating Minimization (AltMin) \cite{Gu2016low}, Alternating Projection (AltProj) \cite{Dutta2018nonconvex}, and Go Decomposition (GoDec) \cite{Zhou2011godec}  applied to $M^{(1)}$ or $M^{(2)}$ with algorithm parameters set to rank $r=1$ and sparsity $s=1$. 
The matrices $M^{(1)}$ and $M^{(2)}$ have values chosen so that the algorithm default initialization causes divergence. While these are particularly simple examples, we would not wish to claim this result is generic in that we do not typically observe divergence for randomly sampled instance of $\alpha, \beta$ in \eqref{eq:lemma2_construction} unless the initialization of the algorithm is adjusted to be near the diverging sequence.
In each case Figure \ref{fig:rpca11-nonconvex} shows the convergence of the residual $\min_{X\in \pR^{m\times n}} \|X-M\|_F$ to zero while the norms of the constituents of $M=L+S$ diverge.  

A line of work suggests adding a regularization term to the objective \cite{Gu2016low,Ge2017no,Zhang2017a}. This leads to bounding the energy of components resulting in the optimization problem to have a global minimum with bounded energies of the constituents. However, the issue of ill-posedness is a more fundamental one; the best rank-$r$ and sparsity-$s$ approximation still has no solution. We observe in Figure \ref{fig:ls11-convex} that energy regularizers result in solutions that are not in the desired space $\LS(r,s)$ for values of $(r,s)=(1,1)$ where the unregularized solution has unbounded energy of its constituents.

\begin{figure}[t]
	\centering
	\begin{subfigure}{0.45\textwidth}
		\includegraphics[width=1\textwidth]{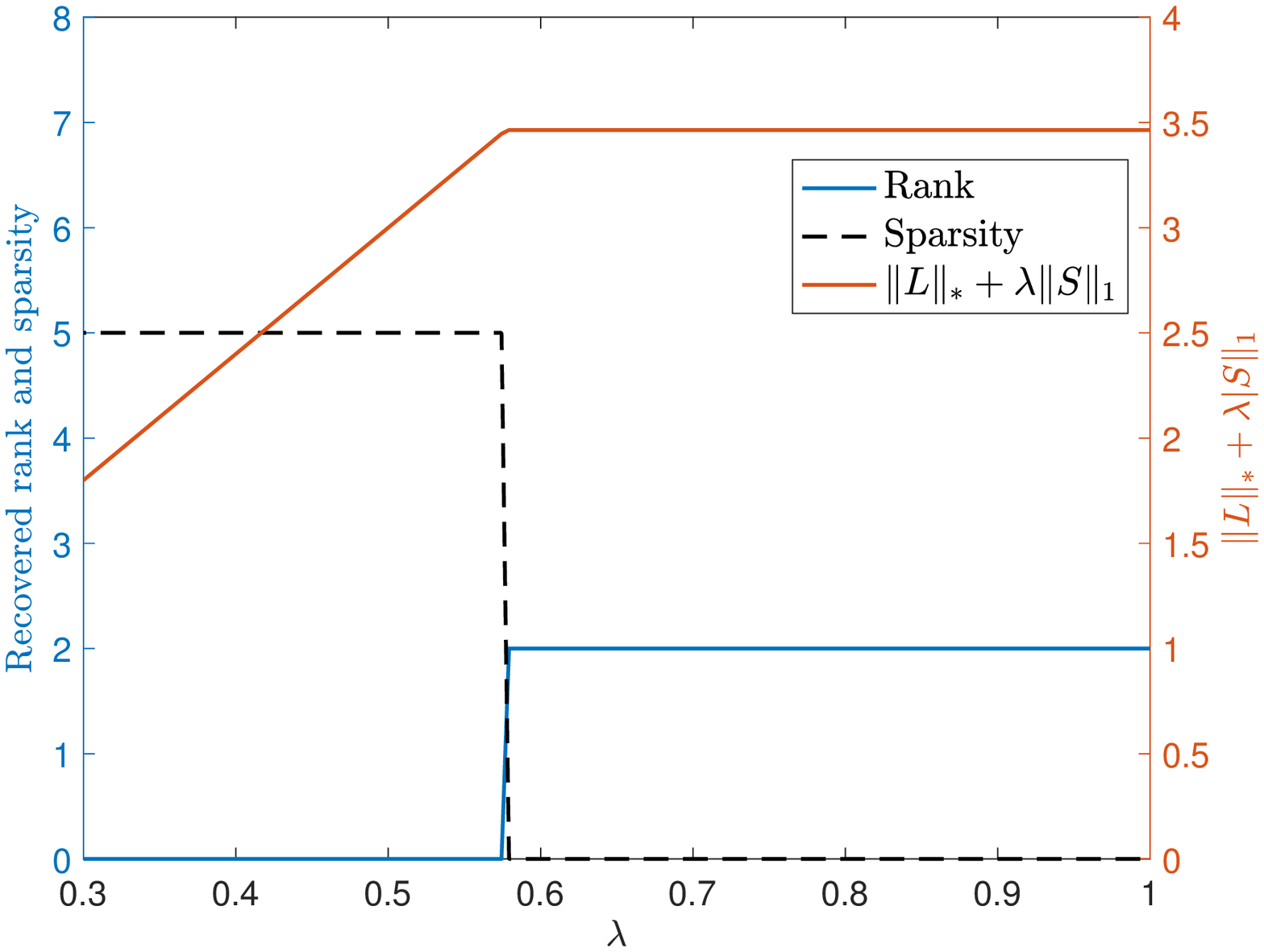}
		\caption{PCP \cite{Candes2011robust} applied to $M^{(1)}$.}
	\end{subfigure}
	\hspace{0.05\textwidth}
	\begin{subfigure}{0.45\textwidth}
		\includegraphics[width=1\textwidth]{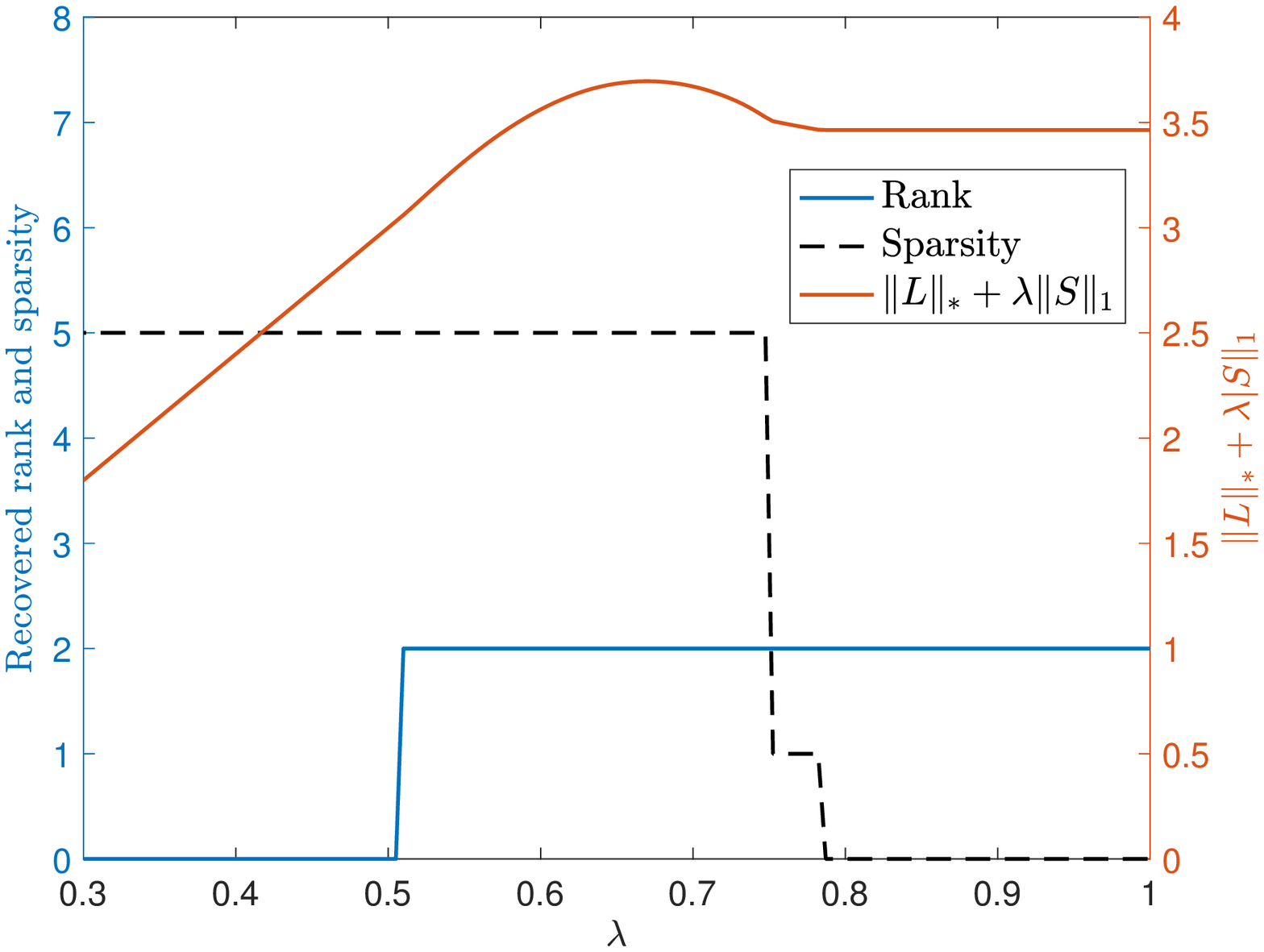}
		\caption{IALM \cite{Lin2010augmented} applied to $M^{(1)}$.}
	\end{subfigure}
	\vspace{-1em}
	\caption{Recovered ranks and sparsities using two convex Robust PCA algorithms applied to $M^{(1)}$ with varying choice of $\lambda$. Both PCP and IALM do not recover the $r=1, s=1$ solution for any $\lambda$. IALM recovers solutions with overspecified degrees of freedom $r=2, s= 5$ for $\lambda$ roughly $1/2$.\label{fig:ls11-convex}}
\end{figure}

The diverging constituents in Figure \ref{fig:rpca11-nonconvex} follow the selected $(r,s)$ for which $M^{(1)},M^{(2)}\not\in\LS(r,s)$ but produce a sequence $L^t+ S^t\in\LS(r,s)$ and $\lim_{t\to\infty} L^t + S^t = M^{(i)}$ but $\|L^t\|_F$ and $\|S^t\|_F$ diverge. This phenomenon does not occur for these matrices if we allow other choices of $(r,s)$. In particular, Alternating Projection method \cite{Netrapalli2014provable} has the rank constraint prescribed and the sparsity constraint is chosen adaptively based on the parameter $\beta$ and the largest singular value of the low-rank component. Such methods, that do not prescribe both $r$ and $s$, are less susceptible to the diverging constituents problem. Methods such as the Alternating Projection \cite{Netrapalli2014provable} typically have a parameter which controls values of $(r,s)$ and can be selected, such that when applied to $M^{(1)}$ it gives a local minimum in $\LS(1,1)$.

Convex relaxations of RPCA such as posed in \eqref{eq:RPCA_problem_convex} do not suffer from the divergence of constituents as shown in Figure \ref{fig:rpca11-nonconvex} due to their explicit minimization of their norms.  However, they suffer from sub-optimal performance.  Figure \ref{fig:ls11-convex} depicts recovered ranks, sparsities and their convex relaxations based on choice for $\lambda$ of $M^{(1)}$ for Principal Component Pursuit by Alternating Directions (PCP) \cite{Candes2011robust} and Inexact Augmented Lagrangian Method (IALM) \cite{Lin2010augmented}.  For both PCP and IALM, as the regularization parameter $\lambda$ is increased from near zero it first produced a solution with $r=0$ and $s=5$, then at approximately $\lambda=1/2$ transitions to solutions with overspecified degrees of freedom $r=2$ and $s=5$, and then for large values of $\lambda$ gives solutions with $r=2$ and $s=0$.  It is interesting to note that for these convex relaxations of RPCA there were no values of $\lambda$ that would produce a solutions with $r=1$ and $s=1$ which are the parameters for which the non-convex RPCA algorithms diverge.  In contrast, the aforementioned non-convex algorithms for RPCA applied to $M^{(1)}$ converge to zero residual with bounded constituents for the rank and sparsity parameters generated by PCP and IALM.

\begin{figure}[t]
	\centering
	\begin{subfigure}{0.45\textwidth}
		\includegraphics[width=1\textwidth]{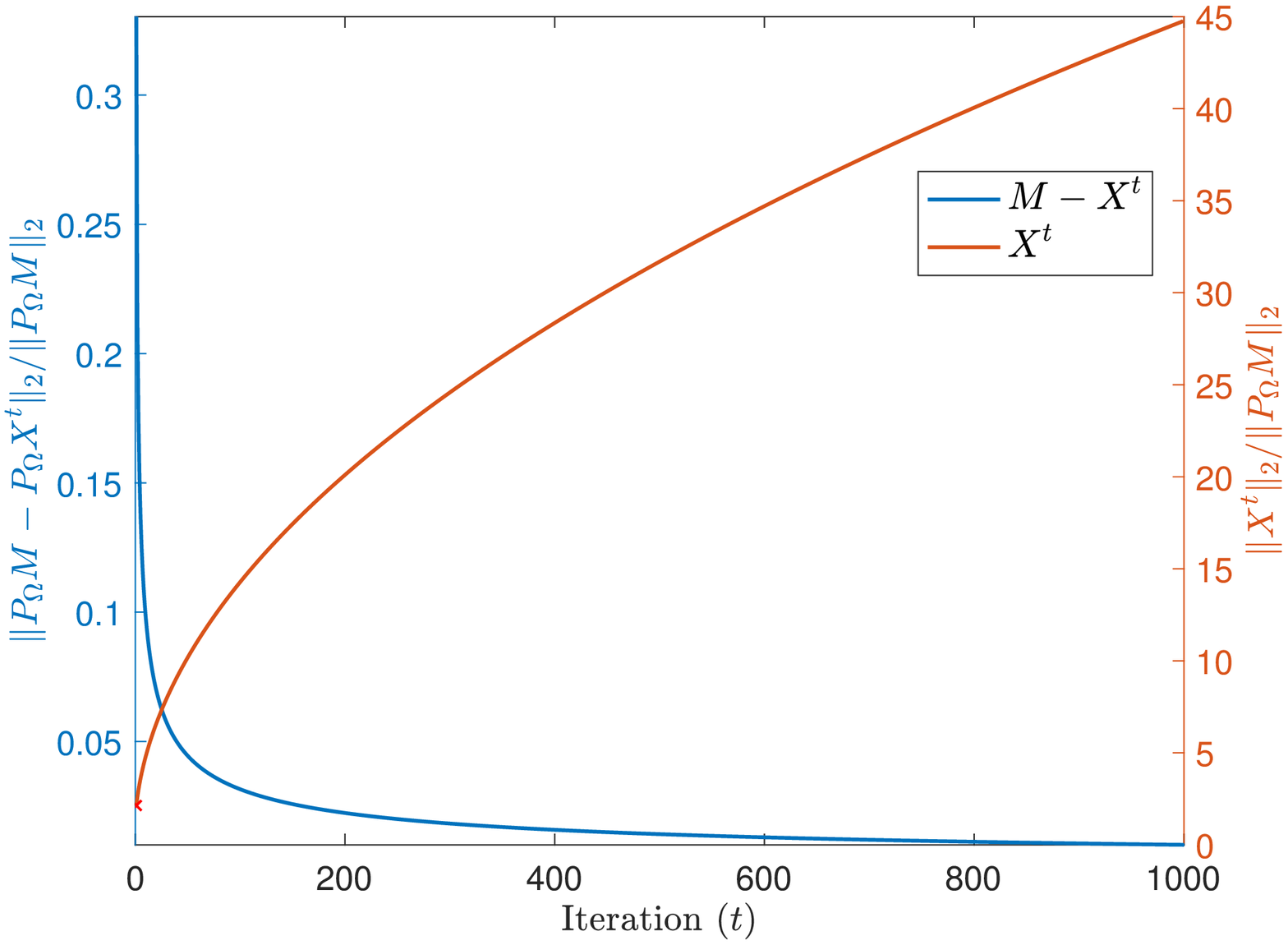}
		\caption{PF \cite{Haldar2009rank} applied to $M^{(1)}$.\label{fig:mc_niht1}}
	\end{subfigure}
	\begin{subfigure}{0.45\textwidth}
		\includegraphics[width=1\textwidth]{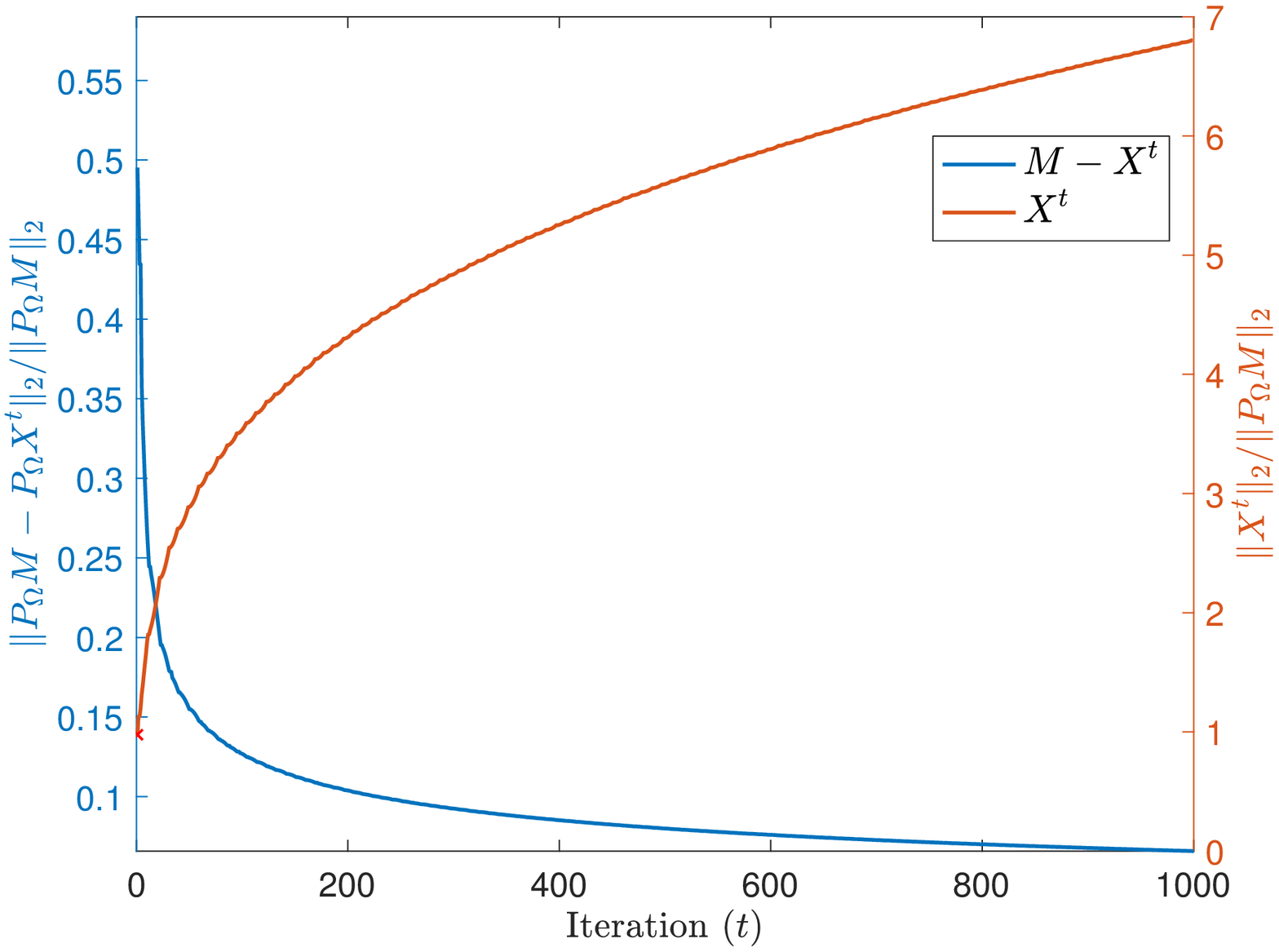}
		\caption{LMaFit \cite{Wen2012solving} applied to $M^{(1)}$.\label{fig:mc_lmafit1}}
	\end{subfigure}
	\begin{subfigure}{0.45\textwidth}
		\includegraphics[width=1\textwidth]{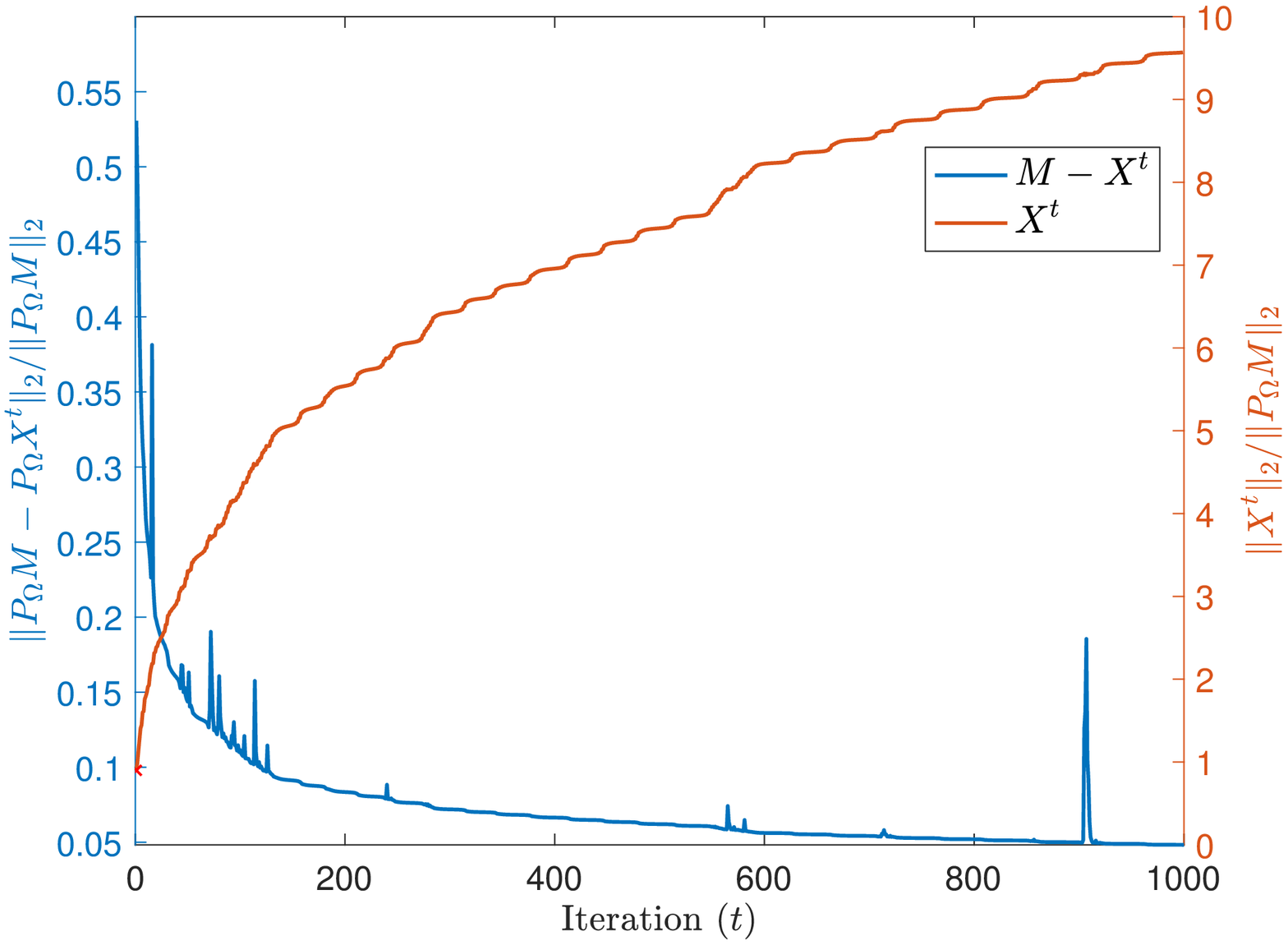}
		\caption{CGIHT with restarts \cite{Blanchard2015cgiht} applied to $M^{(1)}$. \label{fig:mc_cgiht1}}
	\end{subfigure}
	\begin{subfigure}{0.45\textwidth}
		\includegraphics[width=1\textwidth]{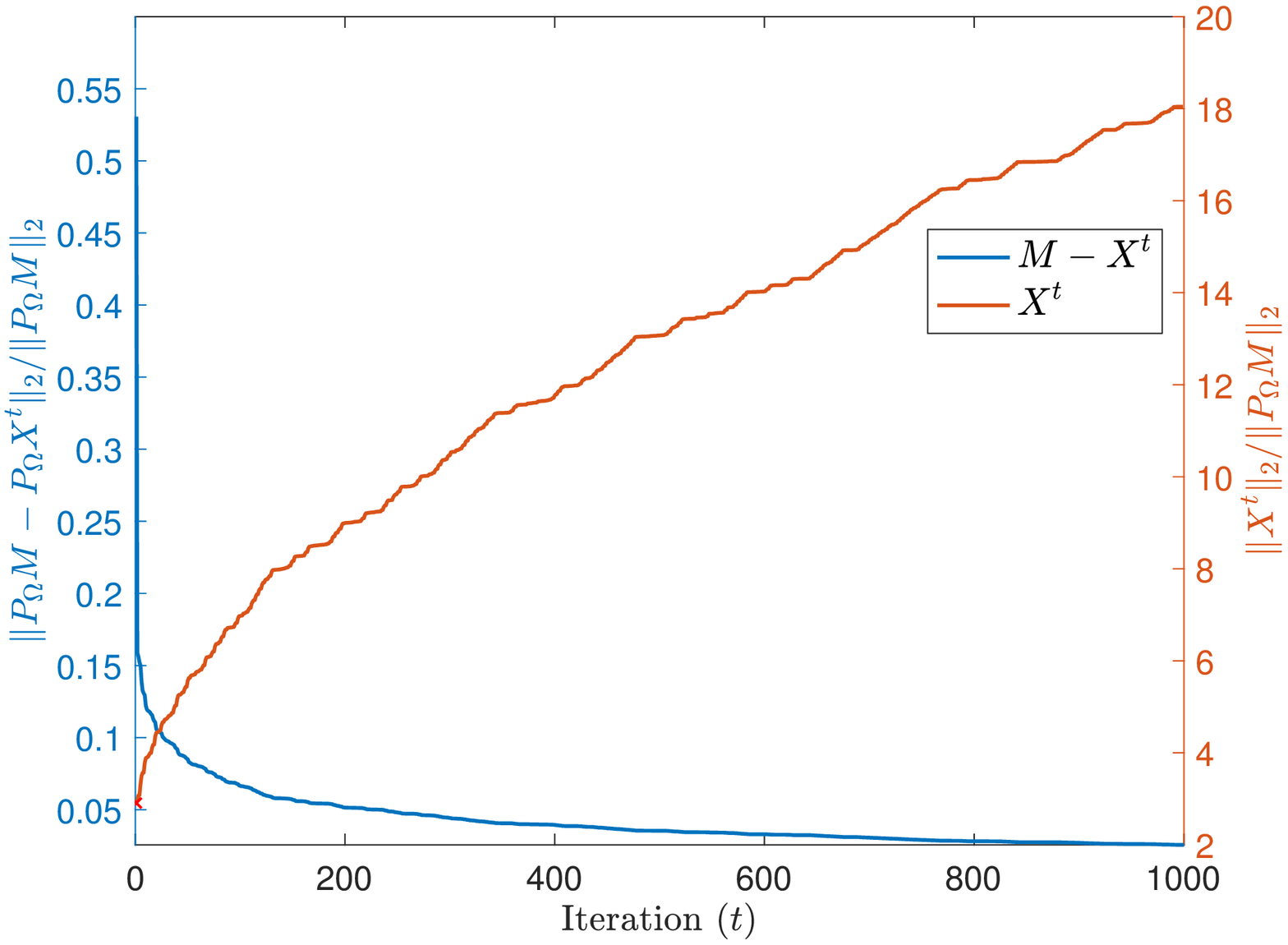}
		\caption{ASD \cite{Tanner2014alternating} applied to $M^{(1)}$.\label{fig:mc_asd1}}
	\end{subfigure}
	\vspace{-1em}
	\caption{Recovery of $M^{(1)}$ given a rank $1$ constraint by four non-convex matrix completion algorithms. Despite the norm of the residual $\| y - \PO (X^t) \|_F$ converging to zero, the norm of the recovered matrix $X^t$ diverges.\label{fig:ls11_mc}}
\end{figure}

Similar to the divergence of the non-convex RPCA algorithms, non-convex matrix completion algorithms applied to $M^{(1)}$ with only the top left, index $(1,1)$, entry missing can diverge\footnote{It is required to provide the algorithm with an initial guess that does not have $0$ as the top left entry.}.\enlargethispage{\baselineskip}
Figure \ref{fig:ls11_mc} depicts the residual error converging to zero and energy of the recovered low-rank matrix diverging for four exemplar non-convex algorithms: Power Factorization (PF) \cite{Haldar2009rank}, Low-Rank Matrix Fitting (LMaFit) \cite{Wen2012solving}, Conjugate Gradient Iterative Hard Thresholding (CGIHT) \cite{Blanchard2015cgiht} and Alternating Steepest Descent (ASD) \cite{Tanner2014alternating}. 

\section{Conclusion}

This work brings to attention an overlooked issue in Robust PCA and matrix completion: that both problems can be ill-posed because the set of low-rank plus sparse matrices is not closed without further conditions being set on the constituent matrices.  It remains to be determined what fraction of the set $L_{m,n}(r,s)$ is open, or similarly what fraction has constituents whose norm exceeds a prescribed threshold to ensure well conditioning; it should be noted that in the case of Tensor CP rank the fraction of the space of tensors with unbounded constituent energy is a positive measure \cite{Silva2008tensor}.  It also remains to determine what is the maximal matrix size $n$, as a function of $(r,s)$, such that the set $\LS_n(r,s)$ is open. We give lower bound of $n(r,s) \geq (r+1)(s+2)$ and $n(r,s) \geq (r+2)^{(3/2)}s^{1/2}$ in Theorem \ref{thm:main_intro} and conjecture the best attainable bound is achieved at $n(r,s) \geq r + (s+1)^{1/2}$ using bounds on maximum matrix rigidity, see Conjecture \ref{conj:maxrigid}.  Moreover, we note that there are references in the literature \cite{Gu2016low, Waters2011sparcs} which reference the use of a restricted isometry property for $\LS_n(r,s)$ in order to prove recovery of RPCA using non-convex algorithms. A consequence of our result is that the lower RIP bound is not satisfied for some $M\in\LS(r,s)$ unless further restrictions are imposed on the constituents such as bounds on the energy of $L$ and $S$ which compose $M$. Another consequence is that there exist semidefinite programs for which the formulation using the Burer-Monteiro low-rank factorization \cite{Burer2003a, Boumal2016nonconvex} will cease to have a solution while the original problem has a solution.

\bibliographystyle{siam}
\bibliography{library}

\begin{thebibliography}{10}

\bibitem{Abdi2010principal}
{\sc H.~Abdi and L.~J. Williams}, {\em {Principal component analysis}}, Wiley
  Interdisciplinary Reviews: Computational Statistics, 2 (2010), pp.~433--459.

\bibitem{Baete2018low}
{\sc S.~H. Baete, J.~Chen, Y.-C. Lin, X.~Wang, R.~Otazo, and F.~E. Boada}, {\em
  {Low rank plus sparse decomposition of ODFs for improved detection of
  group-level differences and variable correlations in white matter}},
  NeuroImage, 174 (2018), pp.~138--152.

\bibitem{Blanchard2015cgiht}
{\sc J.~D. Blanchard, J.~Tanner, and K.~Wei}, {\em {CGIHT: conjugate gradient
  iterative hard thresholding for compressed sensing and matrix completion}},
  Information and Inference,  (2015), pp.~289 -- 327.

\bibitem{Boumal2016nonconvex}
{\sc N.~Boumal, V.~Voroninski, and A.~S. Bandeira}, {\em {The non-convex
  Burer-Monteiro approach works on smooth semidefinite programs}}, in
  Proceedings of the 30th International Conference on Neural Information
  Processing Systems, Barcelona, Spain, 2016, pp.~2765--2773.

\bibitem{Bouwmans2016decomposition}
{\sc T.~Bouwmans, A.~Sobral, S.~Javed, S.~K. Jung, and E.-H. Zahzah}, {\em
  {Decomposition into low-rank plus additive matrices for background/foreground
  separation: A review for a comparative evaluation with a large-scale
  dataset}}, Computer Science Review, 23 (2017), pp.~1--71.

\bibitem{Burer2003a}
{\sc S.~Burer and R.~D. Monteiro}, {\em {A nonlinear programming algorithm for
  solving semidefinite programs via low-rank factorization}}, Mathematical
  Programming, 95 (2003), pp.~329--357.

\bibitem{Cai2010SVT}
{\sc J.-F. Cai, E.~J. Cand{\`{e}}s, and Z.~Shen}, {\em {A singular value
  thresholding algorithm for matrix completion}}, SIAM Journal on Optimization,
  20 (2010), pp.~1956--1982.

\bibitem{Candes2011robust}
{\sc E.~J. Cand{\`{e}}s, X.~Li, Y.~Ma, and J.~Wright}, {\em {Robust principal
  component analysis?}}, Journal of the ACM, 58 (2011), pp.~1--37.

\bibitem{Candes2009exact}
{\sc E.~J. Cand{\`{e}}s and B.~Recht}, {\em {Exact matrix completion via convex
  optimization}}, Foundations of Computational Mathematics, 9 (2009),
  pp.~717--772.

\bibitem{Candes2010thepower}
{\sc E.~J. Candes and T.~Tao}, {\em {The power of convex relaxation:
  near-optimal matrix completion}}, IEEE Transactions on Information Theory, 56
  (2010), pp.~2053--2080.

\bibitem{Chandrasekaran2009ranksparsity}
{\sc V.~Chandrasekaran, S.~Sanghavi, P.~A. Parrilo, and A.~S. Willsky}, {\em
  {Rank-sparsity incoherence for matrix decomposition}}, SIAM Journal on
  Optimization, 21 (2011), pp.~572--596.

\bibitem{Chen2017denoising}
{\sc Y.~Chen, Y.~Guo, Y.~Wang, D.~Wang, C.~Peng, and G.~He}, {\em {Denoising of
  hyperspectral images using nonconvex low rank matrix approximation}}, IEEE
  Transactions on Geoscience and Remote Sensing, 55 (2017), pp.~5366--5380.

\bibitem{Lathauwer2000on}
{\sc L.~{De Lathauwer}, B.~{De Moor}, and J.~Vandewalle}, {\em {On the best
  rank-1 and rank-(R1, R2, ..., RN) approximation of higher-order tensors}},
  SIAM Journal on Matrix Analysis and Applications, 21 (2000), pp.~1324--1342.

\bibitem{Silva2008tensor}
{\sc V.~de~Silva and L.-h. Lim}, {\em {Tensor rank and the ill-posedness of the
  best low-rank approximation problem}}, SIAM Journal on Matrix Analysis and
  Applications, 30 (2008), pp.~1084--1127.

\bibitem{Demmel1997applied}
{\sc J.~W. Demmel}, {\em {Applied Numerical Linear Algebra}}, Society for
  Industrial and Applied Mathematics, 1997.

\bibitem{Drineas2006fast}
{\sc P.~Drineas, R.~Kannan, and M.~W. Mahoney}, {\em {Fast Monte Carlo
  algorithms for matrices II: computing a low-rank approximation to a matrix}},
  SIAM Journal on Computing, 36 (2006), pp.~158--183.

\bibitem{Dutta2018nonconvex}
{\sc A.~Dutta, F.~Hanzely, and P.~Richt{\'{a}}rik}, {\em {A nonconvex
  projection method for robust PCA}},  (2018), pp.~1--24.

\bibitem{Gao2011robust}
{\sc H.~Gao, J.-F. Cai, Z.~Shen, and H.~Zhao}, {\em {Robust principal component
  analysis-based four-dimensional computed tomography}}, Physics in Medicine
  and Biology, 56 (2011), pp.~3181--3198.

\bibitem{Ge2017no}
{\sc R.~Ge, C.~Jin, and Y.~Zheng}, {\em {No spurious local minima in nonconvex
  low rank problems: a unified geometric analysis}}, Proceedings of the 34th
  International Conference on Machine Learning, 70 (2017), pp.~1233--1242.

\bibitem{Gogna2014split}
{\sc A.~Gogna, A.~Shukla, H.~K. Agarwal, and A.~Majumdar}, {\em {Split Bregman
  algorithms for sparse / joint-sparse and low-rank signal recovery:
  application in compressive hyperspectral imaging}}, in 2014 IEEE
  International Conference on Image Processing (ICIP), IEEE, oct 2014,
  pp.~1302--1306.

\bibitem{Jolliffe1988principal}
{\sc C.~Goodall and I.~T. Jolliffe}, {\em {Principal Component Analysis}},
  Springer, 2002.

\bibitem{Gu2016low}
{\sc Q.~Gu and Z.~Wang}, {\em {Low-rank and sparse structure pursuit via
  alternating minimization}}, Proceedings of the 19th International Conference
  on Artificial Intelligence and Statistics, 51 (2016), pp.~600----609.

\bibitem{Gu2014weighted}
{\sc S.~Gu, L.~Zhang, W.~Zuo, and X.~Feng}, {\em {Weighted nuclear norm
  minimization with application to image denoising}}, in 2014 IEEE Conference
  on Computer Vision and Pattern Recognition, no.~2, IEEE, jun 2014,
  pp.~2862--2869.

\bibitem{Haldar2009rank}
{\sc J.~Haldar and D.~Hernando}, {\em {Rank-constrained solutions to linear
  matrix equations using PowerFactorization}}, IEEE Signal Processing Letters,
  16 (2009), pp.~584--587.

\bibitem{Halko2011finding}
{\sc N.~Halko, P.~G. Martinsson, and J.~A. Tropp}, {\em {Finding structure with
  randomness: probabilistic algorithms for constructing approximate matrix
  decompositions}}, SIAM Review, 53 (2011), pp.~217--288.

\bibitem{Han2017clustering}
{\sc E.~Han, P.~Carbonetto, R.~E. Curtis, Y.~Wang, J.~M. Granka, J.~Byrnes,
  K.~Noto, A.~R. Kermany, N.~M. Myres, M.~J. Barber, K.~A. Rand, S.~Song,
  T.~Roman, E.~Battat, E.~Elyashiv, H.~Guturu, E.~L. Hong, K.~G. Chahine, and
  C.~A. Ball}, {\em {Clustering of 770,000 genomes reveals post-colonial
  population structure of North America}}, Nature Communications, 8 (2017),
  p.~14238.

\bibitem{Hitchcock1927the}
{\sc F.~L. Hitchcock}, {\em {The expression of a tensor or a polyadic as a sum
  of products}}, Journal of Mathematics and Physics, 6 (1927), pp.~164--189.

\bibitem{Hitchcock1928multiple}
{\sc F.~L. Hitchcock}, {\em {Multiple invariants and generalized rank of a
  p-way matrix or tensor}}, Journal of Mathematics and Physics, 7 (1928),
  pp.~39--79.

\bibitem{Hsu2016accumulation}
{\sc L.-C. Hsu, C.-Y. Huang, Y.-H. Chuang, H.-W. Chen, Y.-T. Chan, H.~Y. Teah,
  T.-Y. Chen, C.-F. Chang, Y.-T. Liu, and Y.-M. Tzou}, {\em {Accumulation of
  heavy metals and trace elements in fluvial sediments received effluents from
  traditional and semiconductor industries}}, Scientific Reports, 6 (2016),
  p.~34250.

\bibitem{Kumar2014using}
{\sc A.~Kumar, S.~V. Lokam, V.~M. Patankar, and M.~N.~J. Sarma}, {\em {Using
  elimination theory to construct rigid matrices}}, computational complexity,
  23 (2014), pp.~531--563.

\bibitem{Lin2010augmented}
{\sc O.~Kuybeda, G.~A. Frank, A.~Bartesaghi, M.~Borgnia, S.~Subramaniam, and
  G.~Sapiro}, {\em {The Augmented Lagrange Multiplier Method for Exact Recovery
  of Corrupted Low-Rank Matrices}}, Journal of Structural Biology, 181 (2013),
  pp.~116--127.

\bibitem{Kyrillidis2014matrix}
{\sc A.~Kyrillidis and V.~Cevher}, {\em {Matrix recipes for hard thresholding
  methods}}, Journal of Mathematical Imaging and Vision, 48 (2014),
  pp.~235--265.

\bibitem{Lee2010admira}
{\sc K.~Lee and Y.~Bresler}, {\em {ADMiRA: atomic decomposition for minimum
  rank approximation}}, IEEE Transactions on Information Theory, 56 (2010),
  pp.~4402--4416.

\bibitem{Zhou2011godec}
{\sc G.~Liu, Z.~Lin, S.~Yan, J.~Sun, Y.~Yu, and Y.~Ma}, {\em {Robust recovery
  of subspace structures by low-rank representation}}, IEEE Transactions on
  Pattern Analysis and Machine Intelligence, 35 (2013), pp.~171--184.

\bibitem{Luan2014extracting}
{\sc X.~Luan, B.~Fang, L.~Liu, W.~Yang, and J.~Qian}, {\em {Extracting sparse
  error of robust PCA for face recognition in the presence of varying
  illumination and occlusion}}, Pattern Recognition, 47 (2014), pp.~495--508.

\bibitem{Miehlbradt2018data}
{\sc J.~Miehlbradt, A.~Cherpillod, S.~Mintchev, M.~Coscia, F.~Artoni,
  D.~Floreano, and S.~Micera}, {\em {Data-driven body–machine interface for
  the accurate control of drones}}, Proceedings of the National Academy of
  Sciences, 115 (2018), pp.~7913--7918.

\bibitem{Netrapalli2014provable}
{\sc P.~Netrapalli, U.~N. Niranjan, S.~Sanghavi, A.~Anandkumar, and P.~Jain},
  {\em {Non-convex robust PCA}}, Advances in Neural Information Processing
  Systems,  (2014).

\bibitem{Oreifej2013simultaneous}
{\sc O.~Oreifej, X.~Li, and M.~Shah}, {\em {Simultaneous video stabilization
  and moving object detection in turbulence}}, IEEE Transactions on Pattern
  Analysis and Machine Intelligence, 35 (2013), pp.~450--462.

\bibitem{Otazo2015low}
{\sc R.~Otazo, E.~Cand{\`{e}}s, and D.~K. Sodickson}, {\em {Low-rank plus
  sparse matrix decomposition for accelerated dynamic MRI with separation of
  background and dynamic components}}, Magnetic Resonance in Medicine, 73
  (2015), pp.~1125--1136.

\bibitem{Plesa2018multiplexed}
{\sc C.~Plesa, A.~M. Sidore, N.~B. Lubock, D.~Zhang, and S.~Kosuri}, {\em
  {Multiplexed gene synthesis in emulsions for exploring protein functional
  landscapes}}, Science, 359 (2018), pp.~343--347.

\bibitem{Recht2010guaranteed}
{\sc B.~Recht, M.~Fazel, and P.~a. Parrilo}, {\em {Guaranteed minimum-rank
  solutions of linear matrix equations via nuclear norm minimization}}, SIAM
  Review, 52 (2010), pp.~471--501.

\bibitem{Ringner2008what}
{\sc M.~Ringn{\'{e}}r}, {\em {What is principal component analysis?}}, Nature
  Biotechnology, 26 (2008), pp.~303--304.

\bibitem{Sabushimike2016low}
{\sc D.~Sabushimike, S.~Na, J.~Kim, N.~Bui, K.~Seo, and G.~Kim}, {\em {Low-rank
  matrix recovery approach for clutter rejection in real-time IR-UWB
  radar-based moving target detection}}, Sensors, 16 (2016), p.~1409.

\bibitem{Tanner2013normalized}
{\sc J.~Tanner and K.~Wei}, {\em {Normalized iterative hard thresholding for
  matrix completion}}, SIAM Journal on Scientific Computing, 35 (2013),
  pp.~S104--S125.

\bibitem{Tanner2014alternating}
\leavevmode\vrule height 2pt depth -1.6pt width 23pt, {\em {Low rank matrix
  completion by alternating steepest descent methods}}, Applied and
  Computational Harmonic Analysis, 40 (2016), pp.~417--429.

\bibitem{Valiant1977graph}
{\sc L.~G. Valiant}, {\em {Graph-theoretic arguments in low-level complexity}},
  in Mathematical Foundations of Computer Science 1977, J.~Gruska, ed.,
  vol.~53, Springer Berlin Heidelberg, Berlin, Heidelberg, 1977, pp.~162--176.

\bibitem{Vaswani2018rethinking}
{\sc N.~Vaswani, Y.~Chi, and T.~Bouwmans}, {\em {Rethinking PCA for modern data
  sets: theory, algorithms, and applications}}, Proceedings of the IEEE, 106
  (2018), pp.~1274--1276.

\bibitem{Waters2011sparcs}
{\sc A.~E. Waters, A.~C. Sankaranarayanan, and R.~G. Baraniuk}, {\em {SpaRCS:
  recovering low-rank and sparse matrices from compressive measurements}}, in
  Proceedings of the 24th International Conference on Neural Information
  Processing Systems, no.~2, Granada, Spain, 2011, pp.~1089----1097.

\bibitem{Wei2016hyperspectral}
{\sc W.~Wei, L.~Zhang, Y.~Zhang, C.~Wang, and C.~Tian}, {\em {Hyperspectral
  image denoising from an incomplete observation}}, in 2015 International
  Conference on Orange Technologies (ICOT), IEEE, dec 2015, pp.~177--180.

\bibitem{Wen2012solving}
{\sc Z.~Wen, W.~Yin, and Y.~Zhang}, {\em {Solving a low-rank factorization
  model for matrix completion by a nonlinear successive over-relaxation
  algorithm}}, Mathematical Programming Computation, 4 (2012), pp.~333--361.

\bibitem{Woodruff2014sketching}
{\sc D.~P. Woodruff}, {\em {Sketching as a tool for numerical linear algebra}},
  Foundations and Trends in Theoretical Computer Science, 10 (2014),
  pp.~1--157.

\bibitem{Wright2009robust}
{\sc J.~Wright, A.~Yang, A.~Ganesh, S.~Sastry, and {Yi Ma}}, {\em {Robust face
  recognition via sparse representation}}, IEEE Transactions on Pattern
  Analysis and Machine Intelligence, 31 (2009), pp.~210--227.

\bibitem{Xu2017dynamic}
{\sc F.~Xu, J.~Han, Y.~Wang, M.~Chen, Y.~Chen, G.~He, and Y.~Hu}, {\em {Dynamic
  magnetic resonance imaging via nonconvex low-rank matrix approximation}},
  IEEE Access, 5 (2017), pp.~1958--1966.

\bibitem{Yi2016fast}
{\sc X.~Yi, D.~Park, Y.~Chen, and C.~Caramanis}, {\em {Fast algorithms for
  robust PCA via gradient descent}}, Advances in Neural Information Processing
  Systems,  (2016).

\bibitem{Zhang2017a}
{\sc X.~Zhang, L.~Wang, and Q.~Gu}, {\em {A unified framework for low-rank plus
  sparse matrix recovery}}, Proceedings of the 21st International Conference on
  Artificial Intelligence and Statistics, 84 (2018), pp.~1097--1107.

\end{thebibliography}

\end{document}